\documentclass[12pt]{amsart}

\usepackage{amsthm,mathtools}
\usepackage[all]{xy}

\theoremstyle{plain}
\theoremstyle{definition}
\newtheorem{theorem}{Theorem}[section]
\newtheorem{lemma}[theorem]{Lemma}

\newtheorem{proposition}[theorem]{Proposition}
\newtheorem{corollary}[theorem]{Corollary}
\newtheorem{claim}[theorem]{Claim}
\newtheorem{remark}[theorem]{Remark}
\newtheorem{example}[theorem]{Example}
\newtheorem{conjecture}[theorem]{Conjecture}

\usepackage{amssymb,amsmath,tabularx,graphicx,float,placeins}
\usepackage{tikz}
\usetikzlibrary[calc,intersections,through,backgrounds,arrows,decorations.pathmorphing]

\usepackage[normalem]{ulem}

\def\Bbf{\mathbf{B}}

\def\Acal{\mathcal{A}}

\def\Fcal{\mathcal{F}}

\def\Pcal{\mathcal{P}}

\def\Cbb{\mathbb{C}}

\def\Rbb{\mathbb{R}}
\def\Zbb{\mathbb{Z}}

\def\ov{\overline}

\def\pr{\prime}

\def\ra{\rightarrow}

\def\setm{\setminus}

\def\wtil{\widetilde}

\DeclareMathOperator{\Asc}{Asc}  %Ascent set
\DeclareMathOperator{\Bic}{Bic}  %Biclosed sets
\DeclareMathOperator{\Cat}{Cat}  %W-Catalan number
\DeclareMathOperator{\codim}{codim} %codimension
\DeclareMathOperator{\Con}{Con}  %Set of all lattice congruences
\DeclareMathOperator{\con}{con}  %lattice congruence induced by edge-contraction
\DeclareMathOperator{\des}{des}  %descents
\DeclareMathOperator{\Des}{Des}  %descents
\DeclareMathOperator{\GT}{GT}  %Grid-Tamari order
\DeclareMathOperator{\JI}{JI}  %join-irreducibles
\DeclareMathOperator{\lk}{lk}  %link
\DeclareMathOperator{\MI}{MI}  %meet-irreducibles
\DeclareMathOperator{\NC}{NC}  %noncrossing partitions
\DeclareMathOperator{\NF}{NF}  %generators of transitive set
\DeclareMathOperator{\NN}{NN}  %nonnesting partitions
\DeclareMathOperator{\rk}{rk}  %rank
\DeclareMathOperator{\sdes}{sdes}  %Simple descents
\DeclareMathOperator{\SDes}{SDes}  %Simple descents
\DeclareMathOperator{\Seg}{Seg}  %Segments
\DeclareMathOperator{\stars}{star}  %star of face in simplicial complex
\DeclareMathOperator{\SYT}{SYT}  %Simple descents
\DeclareMathOperator{\tr}{tr}
\DeclareMathOperator{\vol}{vol} %volume
\DeclareMathOperator{\Tw}{Tw} %twist map

\begin{document}

\title{Enumerative properties of Grid-associahedra}
\author{Alexander Garver}
\address{Laboratoire de Combinatoire et d'Informatique Math\'ematique,
Universit\'e du Qu\'ebec \`a Montr\'eal}
\email{alexander.garver@lacim.ca}

\author{Thomas McConville}
\address{Department of Mathematics,
Massachusetts Institute of Technology}
\email{thomasmc@mit.edu}

\maketitle

\begin{abstract}

We continue the study of the nonkissing complex that was introduced by Petersen, Pylyavskyy, and Speyer and was studied lattice-theoretically by the second author. We introduce a theory of Grid-Catalan combinatorics, given the initial data of a nonkissing complex, and show how this theory parallels the well-known Coxeter-Catalan combinatorics. In particular, we present analogues of Chapoton's $F$-triangle, $H$-triangle, and $M$-triangle and give combinatorial, lattice-theoretic, and geometric interpretations of the objects defining these. In our Grid-Catalan setting, we prove that Chapoton's $F$-triangle and $H$-triangle identity holds, and we conjecture that Chapoton's $F$-triangle and $M$-triangle identity also holds. As an application, we obtain a bijection between the facets of the nonkissing complex and of the noncrossing complex, which provides a partial solution to an open problem of Santos, Stump, and Welker.

\end{abstract}

\section{Introduction}

Let $\lambda$ be a vertex-induced subgraph of the integer lattice $\Zbb^2$, and orient the vertices edges downward and the horizontal edges to the right. We refer to the graph $\lambda$ as a \emph{shape}. The cone of flows defined by $\lambda$ is the set of nonnegative edge-weightings $x:E(\lambda)\ra\Rbb_{\geq 0}$ such that for each interior vertex, the sum of the weights of incoming edges equals the sum of the weights of outgoing edges. The \emph{flow polytope} $\Pcal_{\lambda}$ is the subset of this cone such that the sum of the weights of edges incident to a source is $1$. The vertices of the flow polytope correspond to directed paths from a source to a sink in the graph $\lambda$. Consequently, triangulations of the polytope correspond to some pure simplicial complexes on paths.

Some unimodular triangulations of flow polytopes for any graph were constructed by Danilov, Karzanov, and Koshevoy \cite{koshevoy_karzanov_danilov:2012coherent} and Postnikov and Stanley \cite{stanley:2000flow}; see also \cite{meszaros_morales_striker:2015flow}. Two special cases of these families of triangulations were investigated by Petersen, Pylyavskyy, and Speyer \cite{petersen.pylyavskyy.speyer:noncrossing}, which they call the \emph{noncrossing complex} and the \emph{nonkissing complex}. They studied a subset of shapes whose flow polytopes are equivalent to polytopes of Gelfand-Tsetlin patterns, fillings of the square regions bounded by $\lambda$ satisfying some inequalities. The two triangulations of these polytopes induce monomial bases of some subalgebras of the Pl\"ucker algebra \cite[Corollary 4.3]{petersen.pylyavskyy.speyer:noncrossing}. Further algebraic and geometric properties of these complexes were studied by Santos, Stump, and Welker \cite{santos.stump.welker:noncrossing}.

The nonkissing complex is realized by a Gorenstein triangulation of the flow polytope, so it is isomorphic to the join of a simplex with a simplicial sphere. The \emph{reduced nonkissing complex} is obtained from the nonkissing complex by deleting its cone points. We construct a complete fan realization of this complex in Section~\ref{sec_fan}, which we call the \emph{Grid-associahedron fan}. We conjecture that it is the normal fan of a polytope, which we call a \emph{Grid-associahedron}.

A key feature of the Grid-associahedron fan is the presence of shards, codimension 1 cones supported by the ridges of the fan. The \emph{shard intersection order} $\Psi^f(\lambda)$ is the lattice of intersections of shards, which was originally introduced by Reading (\cite{reading:2011noncrossing}) for fans defined by simplicial hyperplane arrangements. We prove that the shard intersection order is a graded lattice whose rank generating polynomial is equal to the $h$-polynomial of the nonkissing complex. 

In fact, these structures satisfy a more refined enumerative relationship. Using the fan realization of the nonkissing complex, we define three polynomials in two variables each, called the \emph{$F$-triangle}, the \emph{$H$-triangle}, and the \emph{$M$-triangle}. The names of these polynomials were originally used by Chapoton who defined them for finite Coxeter groups (see \cite{chapoton:2004enumerative}, \cite{chapoton:2006nombre}). The names stand for \emph{face}, \emph{height}, and \emph{M\"obius}, respectively. The word triangle is used to indicate that the coefficients of these polynomials may be put into an $(n+1)\times(n+1)$-matrix in such way that all of the nonzero coefficients appear weakly below the main diagonal. Here $n$ denotes the number of interior vertices of $\lambda$.  The three triangles extend the $f$-polynomial and the $h$-polynomial of the Grid-associahedron fan along with the characteristic polynomial of the shard intersection order. 

One of our main results is an identity between the $F$-triangle and $H$-triangle. In addition, we conjecture an identity between the $F$-triangle and $M$-triangle; see Section~\ref{sec_enumeration}.
These identities are analogous to ones conjectured by Chapoton (\cite{chapoton:2004enumerative}, \cite{chapoton:2006nombre}) and proved by several authors. Indeed, the identities we consider recover Chapoton's in type $A$. For completeness, we recall Chapoton's conjectures in Section~\ref{sec_coxcat}. Our motivation for this work is to present an alternate setting for these enumerative relationships among Coxeter-Catalan objects. We hope to gain a better understanding of these identities by presenting them for a new class of objects.  For more evidence of a wider geometric context for these identities, see \cite{chapoton:2015stokes}.

Part of the beauty of Coxeter-Catalan combinatorics is the abundance of combinatorial structures and the bijections among them. In this spirit, we present three interpretations for each class of objects defining the $F$-triangle, $H$-triangle, and $M$-triangle: a combinatorial interpretation (Section~\ref{sec_nonkissing}), a lattice-theoretic interpretation (Section~\ref{sec_tamari}), and a geometric interpretation (Section~\ref{sec_fan}).

The rest of the paper is structured as follows. To clarify the analogy between Coxeter-Catalan and Grid-Catalan combinatorics, we provide some background on cluster complexes, nonnesting partitions, and noncrossing partitions associated to a Coxeter system in Section~\ref{sec_coxcat}. In Section~\ref{sec_nonkissing}, we introduce the nonkissing complex, the nonfriendly complex, and wide sets of segments, from which we define the $F$-triangle, $H$-triangle, and $M$-triangle, respectively. For some shapes $\lambda$, we provide an equivalent formulation for the $H$-triangle in terms of standard Young tableaux. In Section~\ref{sec_tamari}, we recall the Grid-Tamari order on the facets of the nonkissing complex. Using results from \cite{mcconville:2015lattice}, we give a lattice-theoretic proof that the facets of the nonkissing complex are in bijection with nonfriendly sets and with wide sets. These bijections also have a geometric interpretation, which is covered in Section~\ref{sec_fan}. In Section~\ref{sec_enumeration}, we prove an identity between the $F$-triangle and $H$-triangle and give the conjectured identity between the $F$-triangle and $M$-triangle, mirroring Chapoton's identities in the Coxeter setting.

\section{Coxeter-Catalan combinatorics}\label{sec_coxcat}

In this section, we briefly recall some combinatorial structures that arise in Coxeter-Catalan combinatorics. A thorough account on the development of this subject may be found in \cite[Chapter 1]{armstrong:2009generalized}.

Given a rank $r$ Coxeter system $(W,S)$, the facets of the cluster complex, nonnesting partitions, and noncrossing partitions are each enumerated by $W$-Catalan numbers,
$$\Cat(W)=\prod_{i=1}^r\frac{h+d_i}{d_i}$$
where $h$ is the Coxeter number and $d_1,\ldots,d_r$ are the degrees of the fundamental invariants in $\Cbb[x_1,\ldots,x_r]^W$. Each of these objects were originally defined and studied in type $A$ before being extended to other types. We define each of these objects in turn, and describe some additional enumerative relationships among them.

Let $W$ be a finite real reflection group with root system $\Phi$ and simple roots $\Pi$. A root is \emph{almost positive} if it is either positive or the negation of a simple root. The set $\Phi_{\geq -1}$ of almost positive roots is the ground set of a flag simplicial complex $\Delta(W)$ known as the \emph{(root) cluster complex}. The faces of $\Delta(W)$ are collections of pairwise compatible almost positive roots, as defined in \cite{fomin.zelevinsky:2003systems}. If $W$ is of type $A_{n-1}$, then the cluster complex is isomorphic to the boundary complex of the (dual) associahedron.

The cluster complex arises as a simplicial complex on cluster variables of a finite type cluster algebra. Cluster algebras were introduced by Fomin and Zelevinsky in the study of canonical bases and total positivity in Lie groups, but have since appeared in a wide variety of areas including quiver representations, Teichm\"uller theory, and discrete dynamical systems \cite{fomin.williams.zelevinsky:2016introduction}.

The $F$-triangle \cite{chapoton:2004enumerative} is the polynomial
$$F(x,y)=\sum_{F\in\Delta(W)}x^{|F\cap\Phi^+|}y^{|F\cap(-\Pi)|}.$$
The usual $f$-polynomial of the cluster complex is equal to $F(t,t)$.

For a crystallographic root system $\Phi$, the \emph{root poset} is defined as the poset $(\Phi^+,\leq)$ of positive roots where $\alpha\leq\beta$ if $\beta-\alpha$ is a nonnegative linear combination of simple roots. Postnikov defined the set $\NN(W)$ of \emph{nonnesting partitions} of $W$ to be the antichains of the root poset. Nonnesting partitions may be used to define the $H$-triangle \cite{chapoton:2006nombre},
$$H(x,y)=\sum_{A\in\NN(W)}x^{|A|}y^{|A\cap\Pi|}.$$

We remark that $H(t,1)$ is the usual $h$-polynomial of the cluster complex, which implies
$$H(t+1,1)=t^rF(1/t,1/t),$$
where $r$ is the rank of $W$. A finer relation is given in Equation~\ref{eqn_H_tri_id}.

Noncrossing partitions were introduced by Kreweras \cite{kreweras:partitions} as partitions of a finite subset of $\{1,\ldots,n\}$ arranged in clockwise order on a circle such that the convex hulls of any two blocks do not intersect. This was generalized to all types separately by Bessis \cite{bessis:2003dual} and Brady and Watt \cite{brady.watt:2002k} as follows.

A \emph{Coxeter element} $c$ is the product of each simple generator, taken in any order. To each root $\alpha$ in $\Phi$, we may associate a reflection that fixes a hyperplane and swaps $\alpha$ and $-\alpha$. For $w\in W$, we let $l_T(w)$ be the length of the shortest expression for $w$ as a product of reflections. Coxeter elements are maximal in the \emph{absolute order}, the poset on $W$ where $u\leq v$ if $l_T(u)+l_T(u^{-1}v)=l_T(v)$. The \emph{noncrossing partitions} $\NC(W,c)$ are all elements of $W$ in the interval $[1,c]$ in absolute order. To recover the original definition by Kreweras, we let $c$ be the long cycle $(12\ldots n)$, and replace an element $u\in[1,c]$ with the set of cycles that appear in the cycle decomposition of $u$.

The poset of noncrossing partitions is graded by the length function $l_T$. This allows one to define the $M$-triangle \cite{chapoton:2004enumerative} as the polynomial
$$M(x,y)=\sum_{u\leq v}\mu(u,v)x^{\rk(v)}y^{\rk(u)},$$
where $\mu(u,v)$ is the M\"obius function. 

The following identities were conjectured by Chapoton (\cite{chapoton:2004enumerative}, \cite{chapoton:2006nombre}).
\begin{align}\label{eqn_H_tri_id}
H(x+1,y+1) = x^rF\left(\frac{1}{x},\frac{1+(x+1)y}{x}\right)\\
M(-x,-y/x) = (1-y)^rF\left(\frac{x+y}{1-y},\frac{y}{1-y}\right)
\end{align}

Athanasiadis proved the $F=M$ identity in \cite{athanasiadis:2007some} by calculating the M\"obius function in terms of faces of the cluster complex and by identifying the $h$-polynomial of the cluster complex with the rank generating function of the noncrossing partition lattice. Thiel proved the $F=H$ identity in \cite{thiel:2014h} in a generalized form due to Armstrong \cite{armstrong:2009generalized} by comparing derivatives of each side and using the previously mentioned formula for the $h$-polynomial of the cluster complex.

\section{Grid-Catalan combinatorics}\label{sec_nonkissing}

Let $\lambda$ be a finite induced subgraph of the {$\mathbb{Z}^2$-}lattice. The vertices of $\lambda$ are pairs of integers where $(a,b)$ and $(c,d)$ are adjacent if either $a=c$ and $|b-d|=1$ or $b=d$ and $|a-c|=1$. We orient all of the vertical edges down and the horizontal edges to the right. In Figure~\ref{lambda_ex1}, we show an example a shape $\lambda.$

\begin{figure}
$$\includegraphics[scale=1.5]{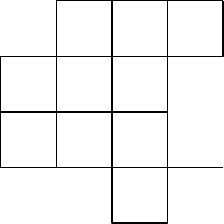}$$
\caption{{A shape $\lambda$.}}
\label{lambda_ex1}
\end{figure}

From the introduction, we recall that the cone of flows is the set of edge-weightings $x:E(\lambda)\ra\Rbb_{\geq 0}$ such that for each vertex $v$ that is neither a sink nor a source, the sum of the incoming weights is equal to the sum of the outgoing weights. The \emph{flow polytope} $\Pcal_{\lambda}$ is the set of edge-weightings in this cone such that the sum of the weights of all edges leaving a source is equal to $1$. For convenience, we define a new graph $\wtil{\lambda}$ from $\lambda$ by identifying all of the vertices with in-degree $\leq 1$ and identifying all of the vertices with out-degree $\leq 1$, and removing all loops that are formed. The flow polytope of $\wtil{\lambda}$ is easily seen to be isomorphic to the flow polytope of $\lambda$.

Label the vertices of $\wtil{\lambda}$ with the integers $1,2,\ldots,N+1$ so that if $\wtil{\lambda}$ has a directed edge $i\ra j$, then $i<j$. Then the unique source has label $1$ and the unique sink has label $N+1$. Let $\Phi=\{e_i-e_j:\ 1\leq i<j\leq N+1\}$ be a root system of type $A_N$. The \emph{Kostant partition function} $K(x)$ evaluated at an element $x$ of the root lattice is equal to the number of ways $x$ may be decomposed into a nonnegative integral sum of positive roots in $\Phi$. The edges of $\wtil{\lambda}$ correspond to a subset $\Phi(\wtil{\lambda})$ of the positive roots. The function $K_{\wtil{\lambda}}(x)$ is the number of ways to write $x$ as a nonnegative integral sum of roots in $\Phi(\wtil{\lambda})$. In unpublished work of Postnikov and Stanley (see e.g., \cite[Section 2.1]{meszaros:2015product}), it was shown that the (normalized) volume of the flow polytope is equal to an evaluation of the Kostant partition function
$$\vol(\Pcal_{\lambda})=K_{\wtil{\lambda}}(0,1,1,\ldots,1,-|V(\wtil{\lambda})|+2).$$

After establishing some basic notation in Section~\ref{subsec_comb_paths_segments}, we recall the \emph{nonkissing complex} from \cite{petersen.pylyavskyy.speyer:noncrossing} in Section~\ref{subsec_comb_nonkissing}. In \cite{petersen.pylyavskyy.speyer:noncrossing}, it is shown for certain shapes that this complex is a regular, unimodular, Gorenstein triangulation of the flow polytope $\Pcal_{\lambda}$, and it is not hard to extend their methods to any shape. 

A second regular, unimodular triangulation called the \emph{noncrossing complex} is more commonly studied in the literature. The facets of the noncrossing complex are in bijection with standard Young tableaux of shape $\lambda$. In Section~\ref{Sec_STY_lambda}, we characterize the facets of this complex for certain \emph{reflected skew shapes} $\lambda$ in terms descent sets of standard Young tableaux of shape $\lambda$. Using this description, we obtain a bijection between facets of the nonkissing complex and facets of the noncrossing complex when $\lambda$ is a rectangle, which provides a partial solution to a problem posed by Santos, Stump, and Welker \cite[Problem 2.22]{santos.stump.welker:noncrossing}. It is still an open question of Santos, Stump, and Welker to define the Grid-Tamari order directly on standard Young tableaux of rectangular shape $\lambda$ to obtain this bijection.

Lastly, we introduce \emph{nonfriendly collections} and \emph{wide sets of segments} in Section~\ref{subsec_comb_nonfriendly}, which we use to describe the (lattice) shard intersection order (see Theorem~\ref{thm_lattice_shard_intersection}).

\subsection{Paths and segments}\label{subsec_comb_paths_segments}

We will use the following terminology to describe points in $\lambda$. We say that $(a,b)$ is \emph{immediately South} (resp., \emph{East}) of $(c,d)$ if $a=c$ and $b=d-1$ (resp., $a=c+1$ and $b=d$). It is an \emph{interior vertex} if all of its four neighbors are in $\lambda$. Vertices in $\lambda$ not in the interior are called \emph{boundary vertices}. We let $V^o$ be the set of interior vertices and $V$ be the set of all vertices of $\lambda$.

A \emph{boundary path} is a sequence of vertices $(v_0,\ldots,v_l),\ l>0$ such that
\begin{itemize}
\item $v_0$ and $v_l$ are boundary vertices,
\item $v_i$ is an interior vertex for $0<i<l$, and
\item $v_i$ is immediately South or East of $v_{i-1}$ for $0<i\leq l$.
\end{itemize}
For the most part, we simply use the word \emph{path} to refer to a boundary path if it causes no confusion. For a path $(v_0,\ldots,v_l)$, we say $v_0$ is the \emph{initial} vertex and $v_l$ is the \emph{terminal} vertex. Using this orientation, we say the path \emph{enters} $v_i$ \emph{from the West} (resp., \emph{North}) if $v_{i-1}$ is immediately West (resp., North) of $v_i$. Similarly, the path \emph{leaves} $v_i$ \emph{to the East} (resp., \emph{South}) if $v_{i+1}$ is immediately East (resp., South) of $v_i$.

It will be useful to define the \emph{transposition} of a shape $\lambda$, denoted $\lambda^{\tr}$. The shape $\lambda^{\tr}$ has vertices of the form $(j, -i)$ where $(i,j)$ is a vertex of $\lambda$ and two vertices of $\lambda^{\tr}$ are connected by an edge if and only if the corresponding vertices of $\lambda$ are connected by an edge.

We frequently consider the special case where $\lambda$ is a rectangular shape. The $k\times(n-k)$ \emph{rectangle} is the shape $\lambda$ with vertex set $\{(i,j):\ 0\leq i\leq n-k,\ 0\leq j\leq k\}$. Boundary paths in $\lambda$ may be extended to paths from $(0,k)$ to $(n-k,0)$, which we may identify with the $k$-element subsets of $[n]:=\{1,\ldots,n\}$.

A \emph{segment} is a sequence of interior vertices $(v_0,\ldots,v_l),\ l\geq 0$ such that $v_i$ is immediately South or East of $v_{i-1}$ for all $i$. We refer to this kind of path as a segment since it may be extended to a boundary path. We say a segment is \emph{lazy} if it only contains one vertex. Two segments $s = (v_0, \ldots, v_l), t = (w_0, \ldots, w_m)$ may be \emph{concatenated} to obtain a new segment $s\circ t := (v_0, \ldots, v_l, w_0, \ldots, w_m)$ if the first vertex of $t$ is immediately South or East of the last vertex of $s$. We let $\Seg(\lambda)$ be the set of all segments supported by $\lambda$. Transposition clearly defines a bijection $\Seg(\lambda) \to \Seg(\lambda^{\tr})$.

If $\lambda$ is a $2\times n$ rectangle, then the segments of $\lambda$ are in natural bijection with the positive roots of a type $A_{n-1}$ root system. Two segments may be concatenated if and only if the sum of their corresponding positive roots is a positive root.

Given segments $t\subseteq s$, we say that $t$ is a \emph{SW-subsegment} of $s$ if
\begin{itemize}
\item $s$ either starts with $t$ or enters $t$ from the North, and
\item $s$ either ends with $t$ or leaves $t$ to the East.
\end{itemize}
A \emph{NE-subsegment} $t\subseteq s$ is defined in the same way, except that $s$ enters $t$ from the West and leaves to the South. Let $A_s$ and $K_s$ be the set of SW-subsegments and NE-subsegments of $s$, respectively. We note that $s$ is the unique common element of $A_s$ and $K_s$. In a similar manner, we define $A_p$ (resp., $K_p$) for a boundary path $p$ to be the set of segments contained in $p$ through which $p$ enters from the North (resp., West) and leaves to the East (resp., South). We observe that $p$ is not in $A_p$ or $K_p$ since boundary paths are not segments.

In the proof of Theorem~\ref{thm_FH}, we treat some of the interior vertices of $\lambda$ as boundary vertices. To simplify the presentation, we typically assume that the boundary vertices are those that have degree at most $3$, but none of our results rely on the assumption that all of the vertices of degree $4$ are considered as interior vertices.

\subsection{The nonkissing complex}\label{subsec_comb_nonkissing}

Boundary paths $p$ and $q$ are \emph{kissing along a common segment} $s$ if
\begin{itemize}
\item $p$ enters $s$ from the West while $q$ enters from the North, and
\item $p$ leaves $s$ to the South while $q$ leaves to the East.
\end{itemize}
We remark that two paths may kiss along several disjoint segments. If they do not kiss along any segment, we say that $p$ and $q$ are \emph{nonkissing}. At times, we will also need to consider when two segments $s, t \in \text{Seg}(\lambda)$ are kissing along a segment. We define this in the same way that we do for boundary paths.

A \emph{simplicial complex} $\Delta$ is a collection of subsets of a given set such that $F\in\Delta$ and $G\subseteq F$ implies $G\in\Delta$. A \emph{facet} is a maximal face of $\Delta$. A complex is \emph{pure} if all of its facets have the same size. If $\Delta$ is pure, then the codimension 1 faces are called \emph{ridges}.

The \emph{nonkissing complex} $\Delta^{NK}(\lambda)$ is the simplicial complex on boundary paths supported by $\lambda$ whose faces consist of pairwise nonkissing paths \cite{mcconville:2015lattice}. Many interesting properties of this triangulation may be derived from the following result.

\begin{theorem}[\cite{petersen.pylyavskyy.speyer:noncrossing}]\label{thm_nonkissing_poly}
The nonkissing complex is a regular, unimodular, Gorenstein triangulation of the flow polytope $\Pcal_{\lambda}$.
\end{theorem}

Petersen, Pylyavskyy, and Speyer proved Theorem~\ref{thm_nonkissing_poly} for a restricted class of shapes, but their arguments may be extended to any shape $\lambda$ in a straight-forward manner.

If $p$ only takes East steps or only takes South steps, we say it is a \emph{horizontal} or \emph{vertical} path, respectively. Horizontal and vertical paths are nonkissing with every other path, so they are cone points in the nonkissing complex. The \emph{reduced nonkissing complex} $\wtil{\Delta}^{NK}(\lambda)$ is the subcomplex of $\Delta^{NK}(\lambda)$ with all horizontal and vertical paths removed. The reduced nonkissing complex is pure of dimension $|V^o|-1$. Furthermore, it is \emph{thin}, which means that every ridge is contained in exactly two facets.

A path $p$ is an \emph{initial path} if it turns at a unique vertex $v$ and $p$ enters $v$ from the West and leaves to the South. The set of all initial paths supported by $\lambda$ is a facet of $\wtil{\Delta}^{NK}(\lambda)$, which we denote by $F_0$. We define the \emph{$F$-triangle} for the reduced nonkissing complex to be the polynomial
$$F(x,y)=\sum_{F\in\wtil{\Delta}^{NK}(\lambda)}x^{|F\setm F_0|}y^{|F\cap F_0|}.$$

\begin{example}
{In Figure~\ref{2x3_ex1}, we show the reduced nonkissing complex where $\lambda$ is a $2\times 3$ rectangle. Two paths are connected by an edge if and only if they lie in a common face. For this complex, we have the following $F$-triangle
$$F(x, y)= 1 + 3x + 2y + 2x^2 + 2xy + y^2.$$}
\end{example}

\begin{figure}
$$\includegraphics[scale=.75]{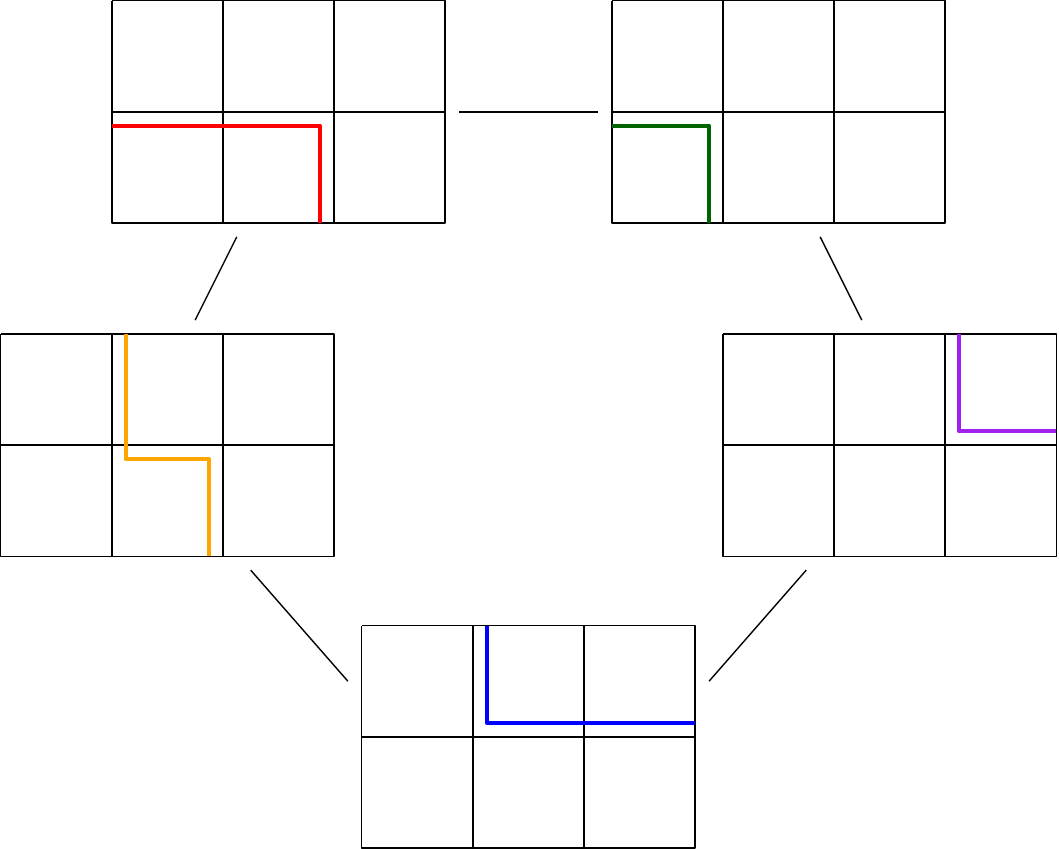}$$
\caption{{The complex $\wtil{\Delta}^{NK}(\lambda)$ where $\lambda$ is a $2\times 3$ rectangle.}}
\label{2x3_ex1}
\end{figure}

\subsection{Standard Young tableaux}\label{Sec_STY_lambda}

A \emph{cell} at $(i,j)$, denoted $\lambda_{(i,j)}$, is the shape whose vertices are $(i,j),\ (i+1,j),\ (i,j+1),$ and $(i+1,j+1)$. A cell at $(i,j)$ is \emph{contained} in $\lambda$ if the vertex-induced subgraph of $\lambda$ on vertices $(i,j), \ (i+1, j), \ (i,j+1)$, and $(i+1,j+1)$ is $\lambda_{(i,j)}$. We say that $\lambda$ is a \emph{reflected skew shape} if the following hold: 
\begin{itemize}
\item if $(i,j)$ and $(i+1,j-1)$ are vertices of $\lambda$, then $(i,j-1)$ and $(i+1,j)$ are vertices of $\lambda$,
\item for any two cells $\lambda_{(a,b)},\ \lambda_{(c,d)}$ contained $\lambda$ there exists two sequences
$$\lambda_{(a_0,b_0)},\ \lambda_{(a_1,b_1)},\ \ldots,\ \lambda_{(a_l,b_l)} \ \ \text{and} \ \  \lambda_{(c_0,d_0)},\ \lambda_{(c_1,d_1)},\ \ldots,\ \lambda_{(c_l,d_l)}$$
consisting of cells contained in $\lambda$ with $(a_0,b_0)=(c_0,d_0),\ (a_l,b_l)=(c_l,d_l),\ (a,b)=(a_i,b_i),\ (c,d)=(c_j,d_j)$ for some $i,j$, and 
\item for any $k  \in \{0, \ldots, l\}$, the cells $\lambda_{(a_k,b_k)}$ and $\lambda_{(c_k,d_k)}$ in these sequences are immediately North or immediately East of $\lambda_{(a_{k-1},b_{k-1})}$ and $\lambda_{(c_{k-1},d_{k-1})}$, respectively.
\end{itemize}

Given a reflected skew shape $\lambda$, let $Q_{\lambda}$ be the poset of cells contained in $\lambda$ where $\lambda_{(a,b)}\leq \lambda_{(c,d)}$ if $a\leq c$ and $b\leq d$.

\begin{proposition}
For a reflected skew shape $\lambda$, the volume of the flow polytope $\Pcal_{\lambda}$ is equal to the number of linear extensions of $Q_{\lambda}$.
\end{proposition}

A linear extension of $Q_{\lambda}$ may be viewed as a filling of the cells of $\lambda$ by the integers $1,\ldots,N$ where each integer appears in a unique cell, and the labels increase to the East and North. Such a filling is called a \emph{standard Young tableau} of shape $\lambda$\footnote{We are using the French conventions for fillings of $\lambda$, which is why he have included the word ``reflected'' in the term reflected skew shape.}. For a standard Young tableau $T$ and $k\in[N]$, we let $T_{\leq k}$ be the restriction of $T$ to the cells with label at most $k$. We view $T_{\leq k}$ as a standard Young tableau whose shape is the union of these cells.

Let $\lambda$ be a reflected skew shape with $N$ cells. We say that $i\in[N-1]$ is a \emph{descent} of a tableau $T$ if $i+1$ is in a cell strictly Northwest of the cell containing $i$. Note that this differs slightly from the usual definition of descent. Typically, one defines $i$ to be a descent if $i+1$ is weakly Northwest of $i$.

Suppose $T$ has a descent at $i$, and let $\lambda_i, \lambda_{i+1}$ be the cells containing $i$ and $i+1$, respectively. To this descent we associate the unique segment $s$ on the boundary of $T_{\leq i+1}$ whose endpoints are the NW-corner of $\lambda_i$ and the SE-corner of $\lambda_{i+1}$. Let $\Des(T)$ be the set of segments corresponding to descents of $T$. A descent at $i$ is \emph{simple} if its associated segment $s$ is lazy and if $s$ is not properly contained in any other descent. Let $\SDes(T)$ be the set of segments associated to simple descents of $T$, and let $\des(T)=|\Des(T)|$ and $\sdes(T)=|\SDes(T)|$. In the next section, we characterize descent sets of tableaux of reflected skew shape.

Let $H^\pr$ be the polynomial
$$H^\pr(x,y)=\sum_{T\in\SYT(\lambda)}x^{\des(T)}y^{\sdes(T)}.$$

\begin{example}
In Figure~\ref{fig_2x3_tableaux}, we show the standard Young tableaux of shape $\lambda$, the $2\times 3$ rectangle. We indicate the descents of these tableaux by drawing the corresponding segments in red. Thus, we have
$$H^\pr(x,y) = 1 + x + 2xy + x^2y^2.$$
More generally, if $\lambda$ is a $2\times n$ rectangle, the segments supported by $\lambda$ correspond to positive roots in a type $A_{n-1}$ root system. The set of segments that may appear as a descent set of some tableau of shape $\lambda$ are precisely the sets of positive roots that may appear in a type $A_{n-1}$ nonnesting partition.
\end{example}

\begin{figure}
$$\includegraphics[scale=1.5]{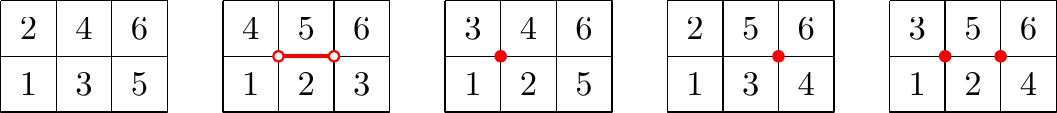}$$
\caption{{The $2\times 3$ standard Young tableau and their descents.}}
\label{fig_2x3_tableaux}
\end{figure}

\begin{example}
{Suppose $\lambda$ is a $3\times 3$ rectangle. In Figure~\ref{fig_3x3_tableau}, we show a standard Young tableau of shape $\lambda$ that gives rise to the term $x^3y$ in $H^\prime(x,y)$.}
\end{example}

\begin{figure}
$$\includegraphics[scale=1.5]{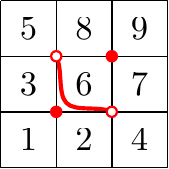}$$
\caption{{A $3\times 3$ standard Young tableau and its descents.}}
\label{fig_3x3_tableau}
\end{figure}

\subsection{The nonfriendly complex}\label{subsec_comb_nonfriendly}

We say two segments $s,t$ are \emph{friendly} along a common subsegment $u$ if both of the following conditions hold:
\begin{itemize}
\item The segment $s$ either starts at $u$ or enters $u$ from the West, and $s$ either ends at $u$ or leaves $u$ to the South.
\item The segment $t$ either starts at $u$ or enters $u$ from the North, and $t$ either ends at $u$ or leaves $u$ to the East.
\end{itemize}

Using the notation from Section~\ref{subsec_comb_paths_segments}, this means $u\in A_s$ and $u\in K_t$ both hold. Two segments are called \emph{nonfriendly} if they are not friendly along any segment. The \emph{nonfriendly complex} $\Gamma^{NF}(\lambda)$ is the simplicial complex on segments supported by $\lambda$ whose faces consist of pairwise nonfriendly segments.

A lazy segment $s$ in a set $S$ of segments is \emph{isolated} if it is not contained in any other member of $S$. Let $\epsilon(S)$ be the set of isolated lazy segments in $S$. We define the \emph{H-triangle} to be the polynomial
$$H(x,y)=\sum_{F\in\Gamma^{NF}(\lambda)}x^{|F|}y^{|\epsilon(F)|}.$$

For the remainder of this section, we fix a reflected skew shape $\lambda$ and give a bijective proof of the following identity.

\begin{theorem}\label{thm_nonnesting_nonfriendly}
When $\lambda$ is a reflected skew shape, the identity $H(x,y) = H^{\pr}(x,y)$ holds.
\end{theorem}

Before proving this theorem, we characterize the descent sets of tableaux.

\begin{lemma}\label{lem_nonnesting_descent_unique}
If $T$ and $T^{\pr}$ are tableaux of shape $\lambda$ with $\text{Des}(T) = \text{Des}(T^\pr)$, then $T=T^{\pr}$.
\end{lemma}

\begin{proof}
Suppose $T$ and $T^{\pr}$ are distinct tableaux with the same descent set, and let $k$ be maximal such that $T_{\leq k-1}=T^{\pr}_{\leq k-1}$. Let $\lambda_{(i,j)}$ and $\lambda_{(i^{\pr},j^{\pr})}$ be the cells containing $k$ in $T$ and $T^{\pr}$, respectively. Without loss of generality, we may assume that $i>i^{\pr}$ and $j<j^{\pr}$ since ${\lambda}_{(i,j)}$ and ${\lambda}_{(i^{\pr},j^{\pr})}$ must be incomparable elements of $Q_{\lambda}$. Let $k^{\pr}$ be minimal such that $k\leq k^{\pr}\leq N$ and $k^{\pr}+1$ occupies a cell weakly Northwest of ${\lambda}_{(i^{\pr},j^{\pr})}$ in $T$. Then there is a descent $s$ in $T$ at $k^{\pr}$ that contains the SE-corner of $\lambda_{(i^{\pr},j^{\pr})}$, which is $(i^\pr+1,j^\pr)$. Since $\text{Des}(T) = \text{Des}(T^\pr)$, $s$ must also be a descent of $T^{\pr}$. Let $(i_1,j_1)$ be the terminal vertex of $s$, and let $k^{\pr\pr}$ be the entry in $T^{\pr}$ at $\lambda_{(i_1,j_1-1)}$. Clearly $k^{\pr\pr}\leq k-1$ is impossible since $T_{\leq k-1}=T^{\pr}_{\leq k-1}$. On the other hand, if $k^{\pr\pr}\geq k+1$, then the segment $s$ cannot start from $(i^\pr+1,j^\pr)$, so it must contain $(i^\pr+1,j^\pr+1)$. This contradicts the minimality assumption on $k^{\pr}$.
\end{proof}

\begin{lemma}\label{lem_nonnesting_descents}
A set of segments $D$ is the descent set of some tableau $T$ if and only if for distinct segments $s,t$ in $D$:
\begin{itemize}
\item $s$ and $t$ do not have the same initial vertex or the same terminal vertex, and 
\item if $u$ is a maximal common subsegment of $s$ and $t$, then $s$ and $t$ are friendly along $u$.
\end{itemize}
Furthermore, no two tableaux of shape $\lambda$ have the same descent set.
\end{lemma}

\begin{proof}
Assume that $D$ is the descent set of a tableau $T$. Suppose $i,j$ are distinct descents of $T$, and let $s$ and $t$ be the segments associated to $i$ and $j$, respectively. Since $i\neq j$, the segments $s$ and $t$ do not have the same initial vertex or the same terminal vertex.

Suppose $s$ and $t$ share at least one vertex, and let $u$ be a maximal common subsegment of each. Then $u$ is on the boundary of $T_{\leq i}$ and $T_{\leq j}$. Without loss of generality, suppose $i<j$. Let $k,l\in[N]$ such that the initial vertex of $u$ is the SE-corner of the cell containing $l$, and the terminal vertex of $u$ is the NW-corner of the cell containing $k$. Since $u$ is a maximal common subsegment of $s$ and $t$, we have $i+1\leq l\leq j+1$ and $i\leq k\leq j$. If $i+1<l$ then $s$ enters $u$ from the West, and if $i+1=l$, then $s$ starts at $u$. Similarly, if $l<j+1$ then $t$ enters $u$ from the North, and if $l=j+1$, then $t$ starts at $u$. On the other end, $s$ either ends at $u$ or leaves $u$ to the South, whereas $t$ either ends at $u$ or leaves $u$ to the East. Hence, $s$ and $t$ are friendly along $u$.

Now assume $D$ is a set of segments such that for distinct $s,t\in D$: $s$ and $t$ do not have the same initial vertex or the same terminal vertex, and if $u$ is a maximal common subsegment of $s$ and $t$, then $s$ and $t$ are friendly along $u$. We construct a tableau $T$ with descent set $D$ as follows.

Fix $k\in[N]$ and suppose $T_{\leq k-1}$ has already been determined. Let $I$ be the set of cells in $T_{\leq k-1}$. Then $I$ is an order ideal of $Q_{\lambda}$. Among the minimal elements of $Q_{\lambda}\setm I$, choose $\lambda_{(i,j)}$ where $i$ is minimal such that:
\begin{itemize}
\item if $s\in D$ such that $s$ contains the SE-corner of $\lambda_{(i,j)}$ but does not contain the NE-corner of $\lambda_{(i,j)}$, then $I$ contains the cell whose NW-corner is the terminal vertex of $s$.
\end{itemize}
We then construct $T_{\leq k}$ by putting $k$ in the cell $\lambda_{(i,j)}$.

We first show that this construction produces a tableau. For a given $k\in[N]$, assume that $T_{\leq k-1}$ is a standard Young tableau. Among the minimal elements of $Q_{\lambda}\setm I$, let $\lambda_{(i,j)}$ be the cell with $i$ maximal. If $s$ is a segment containing the SE-corner of $\lambda_{(i,j)}$, then the cell $\lambda_{(i^{\pr},j^{\pr})}$ whose NW-corner is the terminal vertex of $s$ satisfies $i<i^{\pr}$ and $j>j^{\pr}$. If $\lambda_{(i^{\pr},j^{\pr})}$ is not in $I$, there exists some minimal $\lambda_{(i^{\pr\pr},j^{\pr\pr})}\leq \lambda_{(i^{\pr},j^{\pr})}$ in $Q_{\lambda}\setm I$. Since $j^{\pr\pr}<j$ and $\lambda_{(i,j)}$ is incomparable with $\lambda_{(i^{\pr\pr},j^{\pr\pr})}$, we have $i<i^{\pr\pr}$. This is a contradiction. In particular, the construction of $T_{\leq k}$ is well-defined.

We next show that $D$ is the descent set of $T$. For a given $k\in[N]$, assume that the descents of $T_{\leq k-1}$ are precisely the segments in $D$ connecting cells in $T_{\leq k-1}$. We assume further that any segment in $D$ containing an interior vertex of $T_{\leq k-1}$ is a descent of $T_{\leq k-1}$. Let $\lambda_{(i,j)}$ be the cell in $T$ containing $k$, and let $\lambda_{(i^{\pr},j^{\pr})}$ be the cell containing $k+1$.

If $T$ has a descent at $k$, then $i^{\pr}<i$. In this case, $\lambda_{(i^{\pr},j^{\pr})}$ must have been a minimal element of $Q_{\lambda}\setm I$. Since it was not chosen as the location of $k$, the SE-corner of $\lambda_{(i^{\pr},j^{\pr})}$ must support a segment $s$ in $D$ such that $s$ does not contain the NE-corner of $\lambda_{(i^{\pr},j^{\pr})}$. Since $\lambda_{(i^{\pr},j^{\pr})}$ could be filled with $k+1$, the NW-corner of $\lambda_{(i,j)}$ must be the terminal vertex of $s$, and $s$ must lie along the boundary of $T_{\leq k}$. Hence, $s$ is the segment corresponding to the descent at $k$ in $T$.

Now assume that $D$ contains a segment $s$ whose terminal vertex is the NW-corner of $\lambda_{(i,j)}$. If $s$ does not lie along the boundary of $T_{\leq k}$, let $\lambda_{(i_1,j_1)}$ be the cell with $j_1$ minimal such that $(i_1,j_1)$ is not in $T_{\leq k}$ and $s$ contains both the SE-corner and NE-corner of $\lambda_{(i_1,j_1)}$. Let $\lambda_{(i^{\pr\pr},j^{\pr\pr})}$ be the minimal cell of $Q_{\lambda}\setm I$ such that $j^{\pr\pr}=j_1$, which exists since $\lambda_{(i_1,j_1-1)}$ is in $I$. Since $\lambda_{(i^{\pr\pr},j^{\pr\pr})}$ is not the cell containing $k$, there exists a segment $t$ that contains the SE-corner of $\lambda_{(i^{\pr\pr},j^{\pr\pr})}$ but not the NE-corner, and the cell $\lambda_{(i_2,j_2)}$ whose NW-corner is the terminal vertex of $t$ is not in $T_{\leq k-1}$. By the induction hypothesis, $t$ does not contain any interior vertex of $T_{\leq k-1}$. Hence, $s$ and $t$ meet at the SE-corner of $\lambda_{(i_1,j_1)}$. Call this vertex $v$. Let $u$ be the maximal common subsegment of $s$ and $t$ containing $v$. Then $s$ and $t$ must kiss along $u$. If $t$ ends at $u$, then $s$ must continue to the East after $u$. Since $s$ lies along the boundary of $T_{\leq k-1}$, this means that $(i_2,j_2)$ is already in $T_{\leq k-1}$, a contradiction. Hence, $t$ leaves $u$ to the South. Since $t$ does not contain an interior vertex of $T_{\leq k-1}$, we must have that $s$ ends at $u$. Since, in addition, $\lambda_{(i_2,j_2)}$ is not less than $\lambda_{(i,j)}$ in $Q_{\lambda}$, $t$ must contain both the SW-corner of $\lambda_{(i,j)}$ and the SE-corner of $\lambda_{(i,j)}$. However, by the construction, this would require $\lambda_{(i_2,j_2)}$ to be filled before $\lambda_{(i,j)}$, a contradiction. Therefore, $s$ does lie along the boundary of $T_{\leq k}$. By a similar kissing argument, one can show that $\lambda_{(i^{\pr},j^{\pr})}$ is the cell whose SE-corner is the initial vertex of $s$. In particular, $T$ contains every descent in $D$.\end{proof}

We define a bijection between descent sets of tableaux and nonfriendly sets of segments as follows. Each interior vertex of $\lambda$ may be marked in one of four ways: as a starter, as a closer, as both a starter and a closer, or as neither. In addition, at each edge $e$ between two interior vertices, we place a nonnegative integer weight $\omega_e$. All other edges are given weight $0$. The edge-weighting together with the vertex-marking is \emph{balanced} if at every interior vertex $v$ with incoming edges $e_1,e_2$ and outgoing edges $e_1^{\pr},e_2^{\pr}$, we have

$$\omega_{e_1}+\omega_{e_2}-\omega_{e_1^{\pr}}-\omega_{e_2^{\pr}}=\begin{cases}1 &\mbox{ if }v\mbox{ is a closer but not a starter}\\-1 &\mbox{ if }v\mbox{ is a starter but not a closer}\\0 &\mbox{ otherwise}\end{cases}.$$

Let $X=\Des(T)$ for some tableau $T$. For each edge $e$, let $\omega_e$ be the number of segments in $X$ containing $e$. The initial and terminal vertex of each segment in $X$ are marked as a starter and a closer, respectively. In particular, if $s\in X$ is a lazy segment, its vertex is marked as both a starter and a closer. It is easy to see that this edge-weighting is balanced. Conversely, if $\omega$ is any balanced weighting, then by the characterization in Lemma~\ref{lem_nonnesting_descents}, there is a unique descent set $X$ associated to $\omega$.

On the other hand, if $X$ is a nonfriendly collection, we may again consider the edge-weighting $\omega$ where $\omega_e$ is the number of segments in $X$ containing $e$. The vertices are again marked as starters or closers if they are initial or terminal vertices of a segment in $X$. Once again, this edge-weighting is balanced, and every balanced edge-weighting arises uniquely in this manner. Hence, there is a bijection from descent sets of tableaux to nonfriendly sets that factors through balanced edge-weightings, which we call the \emph{twist map} $\Tw$.

\begin{example}
In Figure~\ref{twist_fig}, we show the twist map $\Tw$ applied to the descent set of a tableau. We show the weights on each edge connecting two interior vertices in blue. All other edge-weights are zero.
\end{example}

\begin{figure}
$$\includegraphics[scale=1.5]{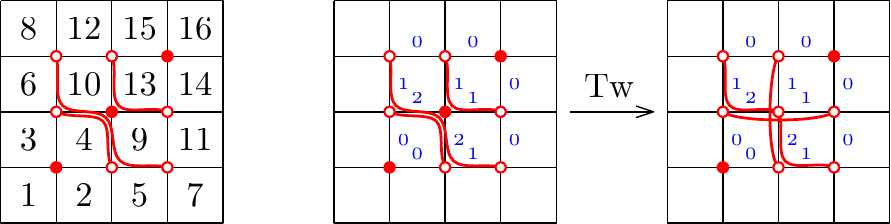}$$
\caption{{The twist map.}}
\label{twist_fig}
\end{figure}

Let $X$ be a descent set. As we do not change the number of starters or closers under $\Tw$, we have $|X|=|\Tw(X)|$. Furthermore, if $s\in X$ is an isolated, lazy segment, then $s$ will be an isolated lazy segment in $\Tw(X)$ as well. We have established the following proposition.

\begin{proposition}\label{prop_twist_map}
The twist map is a bijection between descent sets of tableaux of shape $\lambda$ and nonfriendly sets of segments. Moreover, if $X$ is the descent set of a tableau, then $X$ and $\Tw(X)$ have the same number of segments and the same number of isolated lazy segments.
\end{proposition}

Theorem~\ref{thm_nonnesting_nonfriendly} is an immediate consequence of Proposition~\ref{prop_twist_map} and Lemma~\ref{lem_nonnesting_descents}.

\subsection{Wide sets of segments}\label{subsec_transitive}

We say a subset $T$ of $\Seg(\lambda)$ is \emph{wide} if all of the following conditions hold:
\begin{itemize}
\item For all triples $s,t,u$ with $s\circ t=u$, we have $|T\cap\{s,t,u\}|\neq 2$.
\item If $s,t\in T$ such that $s$ and $t$ are friendly along $u$, then $u\in T$.
\item If $s=s_1\circ s_2\circ s_3$ where $s,s_2\in T$ and either $s_2\in A_s$ or $s_2\in K_s$, then $s_1,s_3\in T$.
\end{itemize}

A collection $T$ of segments is \emph{closed} if whenever $s,t\in T$ and $s\circ t$ is well-defined then $s\circ t\in T$. If $T$ is any set of segments, its \emph{closure} $\ov{T}$ is the smallest closed set containing $T$. Equivalently, the closure $\ov{T}$ is the set of all possible concatenations of segments in $T$. 

Now let $T$ be any wide set. We let $\NF(T)$ be the set of $s\in T$ such that $A_s\cap T=\{s\}=K_s\cap T$. Observe that the first property of wide sets shows that $T$ is closed. This implies that $\overline{\text{NF}(T)} \subset T.$

\begin{proposition}\label{prop_nonfriendly_transitive}
If $X$ is a nonfriendly set of segments, then $\ov{X}$ is a wide set such that $X=\NF(\ov{X})$. Conversely, if $T$ is a wide set, then $\NF(T)$ is nonfriendly and $T=\ov{\NF(T)}$. Consequently, nonfriendly sets of segments are in bijection with wide sets of segments.
\end{proposition}

\begin{proof}

Let $X$ be a nonfriendly set of segments, and put $T=\ov{X}$. We claim that $T$ is wide. Let $s,t,u\in T$ such that $u=s\circ t$. If $s,t\in T$ then $u\in T$ since $T$ is closed. Suppose $s,u\in T$, and let $u=s_{i_1}\circ\cdots\circ s_{i_m}$ for some $i_1,\ldots,i_m\in[l]$. Let $s=s_{j_1}\circ\cdots\circ s_{j_k}$ for some $j_1,\ldots,j_k\in[l]$. Then $s_{i_1}\subseteq s_{j_1}$ or $s_{j_1}\subseteq s_{i_1}$ holds. If $i_1\neq j_1$, then these segments are friendly, a contradiction. Similarly, we have $i_2=j_2$, $i_3=j_3$, etc. Hence, $t=s_{i_{k+1}}\circ\cdots\circ s_{i_m}\in T$, as desired.

Let $s=t_1\circ t_2\circ t_3$ such that $s,t_2\in T$ and either $t_2\in A_s$ or $t_2\in K_s$. We show that $t_1,t_3\in T$. Without loss of generality, we may assume that $t_2$ is in $A_s$. Since $t_2$ is in $T$, there exist $u_1,\ldots,u_l\in X$ such that $t_2=u_1\circ\cdots\circ u_l$. Some factor, say $u_i$, is in $A_{t_2}$, which implies $u_i\in A_s$. Let $s_1,\ldots,s_m\in X$ such that $s=s_1\circ\cdots\circ s_m$. Let $s_j$ be the first segment in this list such that $s_j$ meets $u_i$ and $u_i$ does not extend to the East after the terminal vertex of $s_j$. Then the intersection of $u_i$ and $s_j$ is in $K_{u_i}$ and $A_{s_j}$. Since $u_i$ and $s_j$ are nonfriendly if they are distinct, this forces $u_i=s_j$. We now have $s_1\circ\cdots\circ s_{j-1}=t_1\circ(u_1\circ\cdots\circ u_{i-1})$ and $s_{j+1}\circ\cdots\circ s_m=(u_{i+1}\circ\cdots\circ u_l)\circ t_3$. By the previous paragraph, we deduce that $t_1\in T$ and $t_3\in T$.

Now let $s,t\in T$ such that $s$ and $t$ are friendly along a common subsegment $u$. We prove that $u\in T$ by induction on the sum of the lengths of $s$ and $t$. Since $s$ and $t$ are friendly, at least one of them is not minimal in $T$. Suppose $t=t_1\circ t_2$ where $t_1,t_2\in T$. If $u\in t_i$ for some $i\in\{1,2\}$ then $t_i$ and $s$ are friendly along $u$. In this case, $u\in T$ follows by induction. Otherwise, $u=u_1\circ u_2$ where $t_1$ ends with $u_1$ and $t_2$ starts with $u_2$. Depending on orientation, either $t_1$ is friendly with $s$ along $u_1$ or $t_2$ is friendly with $s$ along $u_2$, or both. Without loss of generality, we assume that $t_2$ is friendly with $s$ along $u_2$, so $u_2\in T$. Let $s=s^{\pr}\circ u_2\circ s^{\pr\pr}$, where $s^{\pr\pr}$ may be empty. Then $s^{\pr}\in T$. Since $s^{\pr}$ and $t_1$ are friendly along $u_1$, we have $u_1\in T$ by induction. Since $T$ is closed, this means $u=u_1\circ u_2\in T$.

Conversely, suppose that $T$ is a wide set. To see that $\text{NF}(T)$ is a nonfriendly set, let $s_1, s_2 \in \text{NF}(T)$ and suppose they are friendly along $u$. Up to reversing the roles of $s_1$ and $s_2$, this means $u \in A_{s_1}$ and $u \in K_{s_2}$. Now since $T$ is wide, $u \in T$. We obtain $u, s_1 \in A_{s_1} \cap T$ and $u, s_2 \in K_{s_2}\cap T$. As $u$ is a proper subsegment of at least one of $s_1$ and $s_2$, this contradicts that $s_1, s_2 \in \text{NF}(T)$.

Next, we show that $T = \overline{{\text{NF}}(T)}$. We have already established that $\overline{\text{NF}(T)} \subset T$ so it remains to show that $T \subset \overline{\text{NF}(T)}$. Suppose $t \in T$. Either $t \in \text{NF}(T)$ or $t \not \in \text{NF}(T)$. We prove the result by induction on the length of $t$ treating the case $t \in \text{NF}(T)$ as the base case. Now, without loss of generality, there exists $s \in A_t \cap T$ where $s \in \text{NF}(T)$ and $s \neq t$ such that either $t = s_1 \circ s \circ s_2$ or $t = s_1 \circ s$ for some non-empty segments $s_1, s_2 \in \text{Seg}(\lambda)$.

Consider the case where $t = s_1\circ s \circ s_2$. As $s, t \in A_t$ and $T$ is wide, we know $s_1, s_2 \in T$. By induction, this implies that $s_1, s_2 \in \overline{\text{NF}(T)}.$ Thus $t \in \overline{\text{NF}(T)}.$

Now consider the case where $t = s_1\circ s$. As $T$ is wide, we know that $|T\cap \{s, s_1, t\}| \neq 2$. We also have that $s, t \in T$ so $s_1 \in T$. By induction, this implies that $s_1 \in \overline{\text{NF}(T)}$. Thus $t \in \overline{\text{NF}(T)}.$\end{proof}

Let $\Psi^{w}(\lambda)$ be the poset of wide sets of segments, ordered by inclusion. In Proposition~\ref{prop_shard_int_graded_lattice}, we prove that $\Psi^{w}(\lambda)$ is a graded lattice. Let $\rk$ be the rank function of this lattice.

Recall that the \emph{M\"obius function} $\mu$ of a poset $P$ is the unique function on closed intervals of $P$ such that for $a\leq b$
$$\sum_{c:\ a\leq c\leq b}\mu(a,c)=\begin{cases}1\ \mbox{if }a=b\\0\ \mbox{if }a\neq b\end{cases}.$$
Letting $\mu$ be the M\"obius function on $\Psi^{w}(\lambda)$, we define the $M$-triangle to be the polynomial
$$M(x,y)=\sum_{\substack{X,Y\in\Psi^{w}(\lambda)\\Y\subseteq X}}\mu(Y,X)x^{\rk(X)}y^{\rk(Y)}.$$

\begin{example}
{Let $\lambda$ be a $2\times n$ rectangle shape. Then $\lambda$ has the set of segments $\{t_{ij}:\ 1\leq i<j\leq n\}$ where $t_{ij}\subseteq t_{kl}$ exactly when $k\leq i<j\leq l$. A partition $\Bbf$ of $[n]$ is \emph{noncrossing} in the classical sense if there do not exist two distinct blocks $B_1,B_2\in\Bbf$ and $i,k\in B_1,\ j,l\in B_2$ such that $i<j<k<l$ holds. Given a partition $\Bbf$, we may consider the set $T$ of segments $t_{ij}$ such that $i$ and $j$ are in the same block of $\Bbf$. Then the partition $\Bbf$ is noncrossing if and only if $T$ is wide.}

{In Figure~\ref{fig_2x3_psi}, we show the lattice $\Psi^{w}(\lambda)$ when $n = 3$. In this case, the $M$-triangle is as follows $$M(x,y) = 1 + 3xy + x^2y^2 - 3x - 3x^2y + 2x^2.$$}
\end{example}

\begin{figure}
$$\includegraphics[scale=1.2]{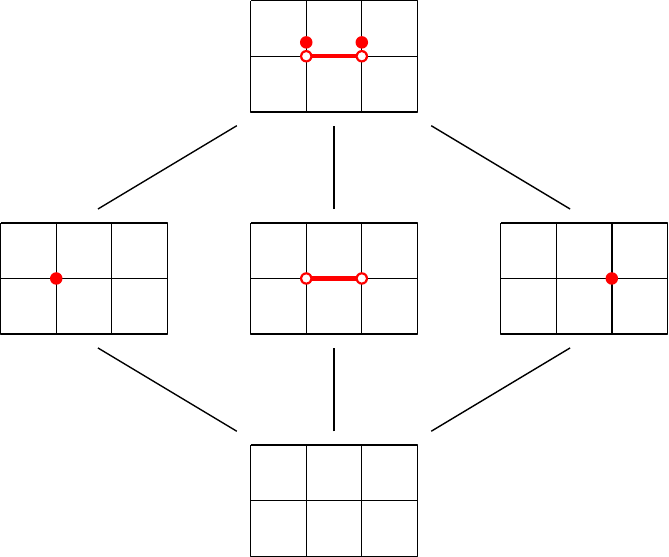}$$
\caption{{The lattice $\Psi^{w}(\lambda)$ where $\lambda$ is a $2\times 3$ rectangle.}}
\label{fig_2x3_psi}
\end{figure}

\section{The Grid-Tamari order}\label{sec_tamari}

The facets of the nonkissing complex form a regular graph where two facets $F,F^{\pr}$ are adjacent if $F\cap F^{\pr}$ is a ridge of $\Delta^{NK}(\lambda)$. Let $F,F^{\pr}$ be adjacent facets and let $p\in F,\ q\in F^{\pr}$ such that $F\setm\{p\}=F^{\pr}\setm\{q\}$. Then $p$ and $q$ kiss along a \emph{unique} maximal subsegment $s$ of $p$ and $q$. We define an orientation and edge-labeling on the graph of facets where $F\stackrel{s}{\ra}F^{\pr}$ if $p$ enters $s$ from the West and leaves to the South and $q$ enters $s$ from the North and leaves to the East. We say that $F^\prime$ may be obtained from $F$ by a \emph{flip}.

In \cite{mcconville:2015lattice}, this directed graph was shown to be the Hasse diagram of a poset called the \emph{Grid-Tamari order}. The Grid-Tamari order, denoted $\GT(\lambda)$, is the poset whose elements are facets of $\widetilde{\Delta}^{NK}(\lambda)$ where for any two facets $F, F^\prime \in \widetilde{\Delta}^{NK}(\lambda)$ we have $F \le F^\prime$ if $F^\prime$ may be obtained from $F$ by a sequence of flips. Furthermore, this poset was proved to be a \emph{congruence-uniform lattice}, which we define in Section~\ref{subsec_gt_lattice}. Congruence-uniformity may be used to define two additional structures, the canonical join complex and the shard intersection order. We will prove that these structures are isomorphic to the nonfriendly complex and the poset of wide sets of segments, respectively.

\subsection{Lattices}\label{subsec_gt_lattice}

In this section, we give some background on lattice theory. The key definitions are the canonical join complex of a semidistributive lattice and the shard intersection order of a congruence-uniform lattice. Throughout this section, we let $(P, \le)$ denote a finite poset.

Given a poset $(P,\leq)$, its \emph{dual}  has the same underlying set, but has the opposite order relations. Many lattice properties come in dual pairs. An \emph{order ideal} $X$ of a poset $P$ is a subset of $P$ such that if $x\leq y$ and $y\in X$ then $x\in X$. An \emph{order filter} of $P$ is an order ideal of the dual of $P$.

Given $x,y\in P$ the \emph{join} $x\vee y$ is the least upper bound of $\{x,y\}$ if it exists. Dually, the \emph{meet} $x\wedge y$ is the greatest lower bound of $\{x,y\}$ if it exists. The poset is a \emph{lattice} if the join and meet of any two elements are defined. An element $j$ of a lattice $L$ is \emph{join-irreducible} if for $x,y\in L$ such that $j=x\vee y$, either $j=x$ or $j=y$. If $L$ is finite, $j$ is join-irreducible exactly when it covers a unique element, which we call $j_*$. A \emph{meet-irreducible} element $m$ is defined dually and is covered by a unique element $m^*$. We let $\JI(L)$ and $\MI(L)$ denote the sets of join-irreducible and meet-irreducible elements of $L$, respectively.

A \emph{join-representation} for an element $x$ is an identity of the form $x=\bigvee A$ for some set of elements $A$. To simplify the language, we say that $A$ is a join-representation of $x$ if $x=\bigvee A$. A \emph{join-representation} $A$ is irredundant if $x>\bigvee B$ for all proper subsets $B\subsetneq A$. We observe that an element $x$ is \emph{join-irreducible} if and only if the only irredundant join-representation of $x$ is $\{x\}$. We partially order irredundant join-representations of $x$, where $A\leq B$ means that for all $a\in A$ there exists $b\in B$ with $a\leq b$. If the set of irredundant join-representations of $x$ has a minimum element, this minimum representation is called the \emph{canonical join-representation} of $x$. The elements in a canonical join-representation are necessarily join-irreducible.

A lattice $L$ is \emph{semidistributive} if
\begin{align*}
x\vee z &=y\vee z\ \ \mbox{implies}\ \ x\vee z=(x\wedge y)\vee z\ \ \mbox{and}\\
x\wedge z &= y\wedge z \ \ \mbox{implies}\ \ x\wedge z=(x\vee y)\wedge z 
\end{align*}
for all $x,y,z\in L$. Equivalently \cite[Theorem 2.24]{freese.jezek.nation:1995free}, a lattice is semidistributive if and only if every element admits a canonical join-representation and a canonical meet-representation, defined dually.

If a set $A$ is a canonical join-representation of some element, then so is any subset of $A$. Hence, the set of canonical join-representations is the set of faces of a simplicial complex, known as the \emph{canonical join complex}. In \cite{barnard:2016canonical}, it was shown that for any finite semidistributive lattice $L$, the canonical join complex is flag. For example, the faces of the canonical join complex of the lattice of order ideals of a finite poset $P$ is the set of antichains of $P$. The canonical join complex of the weak order of type A is a simplicial complex of noncrossing arc diagrams \cite{reading:2015noncrossing}. This complex contains the canonical join complex of the Tamari lattice. In Figure~\ref{fig_2x3_cjc}, we show the canonical join complex of $\GT(\lambda)$ where $\lambda$ is a $2\times 3$ rectangle. Equivalently, Figure~\ref{fig_2x3_cjc} is the canonical join complex of the rank 2 Tamari lattice.

\begin{figure}
$$\includegraphics[scale=1.2]{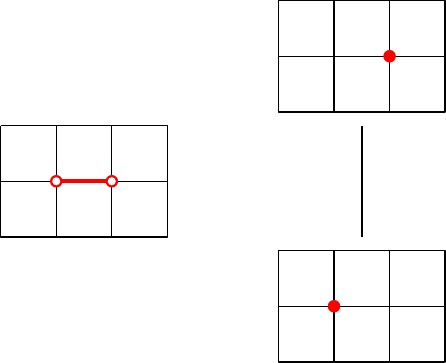}$$
\caption{{The canonical join complex where $\lambda$ is a $2\times 3$ rectangle.}}
\label{fig_2x3_cjc}
\end{figure}

Let $L$ be a semidistributive lattice. If $x=\bigvee A$ is the canonical join-representation of some element $x\in L$, then there is a bijection between lower covers of $x$ and the elements of $A$ \cite{barnard:2016canonical}. Conversely, if $x=\bigwedge B$ is a canonical meet-representation, then the upper covers of $x$ are in bijection with the elements of $B$.

An equivalence relation $\Theta$ on a lattice $L$ is a \emph{lattice congruence} if whenever $x\equiv y\mod\Theta$, we have $x\wedge z\equiv y\wedge z\mod\Theta$ and $x\vee z\equiv y\vee z\mod\Theta$. For $x\in L$, we let $[x]=[x]_{\Theta}$ be the $\Theta$-equivalence class of $x$. The \emph{quotient lattice} $L/\Theta$ is the lattice of $\Theta$-equivalence classes where $[x]\vee[y]=[x\vee y]$ and $[x]\wedge[y]=[x\wedge y]$ for $x,y\in L$. The following characterization of lattice congruences of finite lattices is well-known; see e.g. \cite[Proposition 9-5.2]{reading:2016lattice}.

\begin{lemma}\label{lem_lattice_congruence}
An equivalence relation $\Theta$ on a lattice $L$ is a lattice congruence if and only if
\begin{itemize}
\item the equivalence classes of $\Theta$ are all closed intervals of $L$, and
\item the maps $\pi^{\uparrow}$ and $\pi_{\downarrow}$ taking an element of $L$ to the largest and smallest elements of its $\Theta$-equivalence class are both order-preserving.
\end{itemize}
\end{lemma}

In particular, if $x$ covers $y$ in $L$, then either $[x]=[y]$ or $[x]$ covers $[y]$ in $L/\Theta$. A stronger result holds \cite[Proposition 2.2]{reading:lattice_congruence}, which we use to determine the lattice congruences of the Grid-Tamari order.

\begin{lemma}\label{lem_lattice_quotient_cover}
Let $L$ be a finite lattice with lattice congruence $\Theta$. If $x$ is the minimum element in its $\Theta$-equivalence class, then for each $y\in L$ covered by $x$, the class $[y]$ is covered by $[x]$ in $L/\Theta$. Furthermore, this is a bijection between lower covers of $x$ and lower covers of $[x]$. Dually, if $x$ is the maximum element in its $\Theta$-equivalence class, then there is a similar bijection between the upper covers of $x$ and the upper covers of $[x]$.
\end{lemma}

Given a lattice $L$, its set of lattice congruences $\Con(L)$ forms a distributive lattice under refinement order. Hence when $L$ is finite, $\Con(L)$ is isomorphic to the lattice of order ideals of $\JI(\Con(L))$. If $y$ covers $x$, we write $\con(x,y)$ for the minimal lattice congruence in which $x\equiv y\ (\con(x,y))$ holds.

For any finite lattice $L$ with lattice congruence $\Theta$, we have
$$\Theta=\bigvee_{\substack{j\in \JI(L)\\ j\equiv j_*\mod\Theta}}\con(j_*,j).$$
Hence, the join-irreducible congruences are always of the form $\con(j_*,j)$ for some $j\in \JI(L)$.  A finite lattice $L$ is \emph{congruence-uniform} (or \emph{bounded}) if
\begin{itemize}
\item the map $j\mapsto\con(j_*,j)$ is a bijection from $\JI(L)$ to $\JI(\Con(L))$, and
\item the map $m\mapsto\con(m,m^*)$ is a bijection from $\MI(L)$ to $\MI(\Con(L))$.
\end{itemize}
Alternatively, finite congruence-uniform lattices may be characterized as homomorphic images of free lattices with bounded fibers or as lattices constructible from the one-element lattice by a sequence of interval doublings \cite{day:congruence}.

For $x\in L$, let $\psi(x)$ be the set
$$\{\con(w,z):\ \bigwedge_{i=1}^l y_i\leq w\lessdot z\leq x\},$$
where $y_1,\ldots,y_l$ are the elements covered by $x$. The \emph{(lattice) shard intersection order} $\Psi^l(L)$ is the collection of sets $\psi(x)$, ordered by inclusion. For a congruence-uniform lattice $L$, the map $x\mapsto\psi(x)$ is a bijection between $L$ and $\Psi^l(L)$. The shard intersection order derives its name from some geometric examples, which we recall in Section~\ref{sec_shard}. This formulation in lattice-theoretic terms was given by Reading \cite{reading:2016lattice} and was used in \cite{garver.mcconville:2016oriented} to give a correspondence between noncrossing tree partitions and some partial triangulations of a polygon.

\subsection{Biclosed sets}\label{subsec_gt_biclosed}

Recall from Section~\ref{subsec_transitive} that a collection $X$ of segments is \emph{closed} if whenever $s,t\in X$ and $s\circ t$ is well-defined then $s\circ t\in X$. We say $X$ is \emph{biclosed} if it is closed and its complement $\Seg(\lambda)\setm X$ is closed. The set $\Bic(\lambda)$ of biclosed sets of segments forms a poset under inclusion. That is, if $X,Y\in\Bic(\lambda)$, we set $X\leq Y$ if $X\subseteq Y$.

For a general closure space, the poset of biclosed sets may not be a lattice. Although the closure of the union of two sets $X,Y$ is necessarily closed, its complement may not be. Furthermore, even if the poset of biclosed sets is a lattice, there may be too few biclosed sets to be useful. For the closure space on segments, we proved in \cite{mcconville:2015lattice} that there are enough biclosed sets in the sense that $\Bic(\lambda)$ is graded by the cardinality function. Moreover, $\Bic(\lambda)$ is a congruence-uniform lattice.

\begin{theorem}\label{thm_bic_prop}\cite{mcconville:2015lattice}
The poset of biclosed sets has the following properties. These three properties together imply that $\Bic(\lambda)$ is a congruence-uniform lattice.
\begin{enumerate}
\item The poset $\Bic(\lambda)$ is graded, with rank function $X\mapsto|X|$.
\item\label{thm_bic_prop_2} The poset $\Bic(\lambda)$ is a lattice where
$$X,Y,W\in\Bic(\lambda),\ W\subseteq X\cap Y\ \ \mbox{implies}\ \ W\cup\ov{(X\cup Y)\setm W}\in\Bic(\lambda).$$
\item If $u\in\ov{\{s,t\}}\setm\{s,t\}$, then $u=s\circ t$.
\end{enumerate}
\end{theorem}

We note that taking $W=\emptyset$ in \ref{thm_bic_prop_2}, we have $\ov{X\cup Y}$ is biclosed whenever $X$ and $Y$ are biclosed. Since $\ov{X\cup Y}$ is the smallest closed set containing both $X$ and $Y$, this set is the join of $X$ and $Y$. We can also calculate the meet of two biclosed sets as follows.

\begin{lemma}\label{Lemma_meet_in_bic}
For $X, Y \in \text{Bic}(\lambda)$, one has $X \wedge Y = (X^c \vee Y^c)^c.$
\end{lemma}
\begin{proof}
Let $s \in X \wedge Y$. It follows that $s \in X \cap Y$. Now $s \not \in X^c\cup Y^c \subset {X^c \vee Y^c}.$ This implies that $s \in (X^c \vee Y^c)^c.$

To prove the opposite inclusion, observe that $$X \wedge Y = \displaystyle \bigvee_{\begin{array}{c}\small\text{$Z \in \text{Bic}(\lambda)$}\\ \small \text{$Z \subset X, Z\subset Y$} \end{array}} Z.$$ Now notice that if $s \in (X^c \vee Y^c)^c$, then $s \not \in X^c$ and $s \not \in Y^c$. This implies that $s \in X$ and $s \in Y$. Thus $(X^c \vee Y^c)^c \subset X$ and $(X^c \vee Y^c)^c \subset Y$. Since $(X^c \vee Y^c)^c \in \text{Bic}(\lambda)$, it follows that $(X^c \vee Y^c)^c$ is a joinand in the join representation of $X \wedge Y$ shown above. We obtain that $(X^c \vee Y^c)^c \subset X \wedge Y$.\end{proof}

The Grid-Tamari order $\GT(\lambda)$ is isomorphic to both a quotient lattice and a sublattice of $\Bic(\lambda)$. We recall the quotient lattice and sublattice maps from \cite[Section 8]{mcconville:2015lattice}. Define the quotient map $\eta:\Bic(\lambda)\ra\GT(\lambda)$ by extending each vertical edge of $\lambda$ to a boundary path as follows. For a set of segments $X$, let $\eta(X)$ be the set of paths
$$\eta(X)=\{p_e:\ e \mbox{ is a vertical edge of }\lambda\}$$
where $p_e=(v_0,\ldots,v_l)$ is the unique boundary path such that $e=(v_{j-1},v_j)$ for some $j$, and the following two conditions hold.
\begin{itemize}
\item For $1\leq i\leq j-1$, the vertex $v_{i-1}$ is North of $v_i$ if $(v_i,\ldots,v_{j-1})\in X$ and $v_{i-1}$ is West of $v_i$ if $(v_i,\ldots,v_{j-1})\notin X$.
\item For $j\leq k<l$, the vertex $v_{k+1}$ is East of $v_k$ if $(v_j,\ldots,v_k)\in X$ and $v_{k+1}$ is South of $v_k$ if $(v_j,\ldots,v_k)\notin X$.
\end{itemize}

When restricted to a biclosed set of segments, $\eta(X)$ is a collection of nonkissing paths. Upon removing the vertical paths in this set, $\eta(X)$ is a facet of $\wtil{\Delta}^{NK}(\lambda)$. Moreover, the map $\eta$ is surjective and preserves the lattice operations, so it is a lattice quotient map.

Let $\Theta$ be the equivalence relation on biclosed sets where $X\equiv Y\mod\Theta$ if $\eta(X)=\eta(Y)$. Since $\eta$ is a lattice map, $\Theta$ is a lattice congruence. For $X\in\Bic(\lambda)$, the minimum biclosed set $\Theta$-equivalent to $X$ is the set
$$X^{\downarrow}=\{s\in X:\ A_s\subseteq X\},$$
where $A_s$ is the set of SW-subsegments of $s$ defined in Section~\ref{subsec_comb_paths_segments}.

Let $\phi:\GT(\lambda)\ra\Bic(\lambda)$ be the function where $\phi(F)=\bigvee_{p\in F}A_p$. As $A_p$ is minimal in its $\Theta$-equivalence class, so is the join of any set of elements of the form $A_p$. It was proved in \cite[Section 8]{mcconville:2015lattice} that $\phi$ is an embedding of the poset $\GT(\lambda)$ in $\Bic(\lambda)$. Hence, $\phi$ identifies $\GT(\lambda)$ with a join-subsemilattice of $\Bic(\lambda)$. We claim that this map also preserves meets, so it is a sublattice map. We prove the following equivalent lemma.

\begin{lemma}\label{lem_sublattice}
If $X$ and $Y$ are $\Theta$-minimal biclosed sets of segments, then so is $X\wedge Y$.
\end{lemma}

\begin{proof}
Let $\Theta^{\tr}$ be the congruence on $\Bic(\lambda^{\tr})$ induced by the map $\eta:\Bic(\lambda^{\tr})\ra\GT(\lambda^{\tr})$ ({see Section~\ref{subsec_comb_paths_segments}}). Via the natural bijection on segments $\Seg(\lambda)\ra\Seg(\lambda^{\tr})$, the complement of a $\Theta$-minimal set $X$ in $\Bic(\lambda)$ is a $\Theta^{\tr}$-maximal element of $\Bic(\lambda^{\tr})$.

Let $X,Y$ be $\Theta$-minimal elements of $\Bic(\lambda)$, and let $X^{\pr}=\Seg(\lambda^{\tr})\setm X^{\tr},\ Y^{\pr}=\Seg(\lambda^{\tr})\setm Y^{\tr}$. Then $X^{\pr}$ and $Y^{\pr}$ are $\Theta^{\tr}$-maximal elements. The join $X^{\pr}\vee Y^{\pr}$ is equal to $\ov{X^{\pr}\cup Y^{\pr}}$. We claim that this set is $\Theta^{\tr}$-maximal. Let $s,t\in\Seg(\lambda^{\tr})$ such that $t\in K_s$ and $t\in\ov{X^{\pr}\cup Y^{\pr}}$. Then $t=t_1\circ\cdots\circ t_l$ where each $t_i$ is in $X^{\pr}\cup Y^{\pr}$. Then $t_i\in K_t$ for some $i$, which means $t_i\in K_s$. Since $t_i$ is in $X^{\pr}$ or $Y^{\pr}$ and both sets are $\Theta^{\tr}$-maximal, we have $s\in X^{\pr}\cup Y^{\pr}$. Consequently, the join of $X^{\pr}$ and $Y^{\pr}$ is $\Theta^{\tr}$-maximal, so the meet of $X$ and $Y$ is $\Theta$-minimal.
\end{proof}

\begin{proposition}\label{sub_quot_phi_eta}
The lattice $\GT(\lambda)$ is both a sublattice and a quotient lattice of $\Bic(\lambda)$. Moreover, given any facet ${F} \in \GT(\lambda)$ and any biclosed set $X \in \Bic(\lambda)$, one has $\eta\circ \phi(F) = F$ and $\phi\circ \eta(X) = X^{\downarrow}.$
\end{proposition}

The first sentence of the proposition is a consequence of Lemma~\ref{lem_sublattice} and \cite[Theorem 8.12]{mcconville:2015lattice}. The remaining assertions follow from \cite[Claim 8.7, 8.8]{mcconville:2015lattice}.

\subsection{Canonical join complex}\label{subsec_gt_joinreps}

The main result in this section is Theorem~\ref{thm_NF_CJ}, which states that the canonical join complex of the Grid-Tamari order is isomorphic to the nonfriendly complex.

\begin{lemma}\label{cl_join-irreducible}
For $s\in\Seg(\lambda)$, $\eta(A_s)$ is join-irreducible.
\end{lemma}

\begin{proof}
For $s\in\Seg(\lambda)$, if $\eta(A_s)=F\vee F^{\pr}$, then by Proposition~\ref{sub_quot_phi_eta}
$$A_s=\phi\circ\eta(A_s)=\phi(F\vee F^{\pr})=\phi(F)\vee\phi(F^{\pr}).$$
Since $A_s$ is join-irreducible in $\Bic(\lambda)$, we deduce that $A_s=\phi(F)$ (or $A_s=\phi(F^{\pr})$), so $\eta(A_s)=F$.  Hence, $\eta(A_s)$ is join-irreducible.
\end{proof}

Let $f:\Seg(\lambda)\ra \JI(\GT(\lambda))$ where $f(s)=\eta(A_s)$.

\begin{lemma}\label{cl_join-irreducible_bijection}
The function $f$ is a bijection.
\end{lemma}

\begin{proof}
Clearly $f$ is injective, since if $\eta(A_s)=\eta(A_t)$ for some $s,t\in S$, then $A_s=\phi\circ\eta(A_s)=\phi\circ\eta(A_t)=A_t$ by Proposition~\ref{sub_quot_phi_eta}.

Let $F$ be a join-irreducible of $\GT(\lambda)$.  Then $\phi(F)=\bigvee_{p\in F}A_p$, so by Proposition~\ref{sub_quot_phi_eta} we have
$$F=\eta\circ\phi(F)=\eta\left(\bigvee_{p\in F} A_p\right)=\bigvee_{p\in F}\eta(A_p).$$
Since $F$ is join-irreducible, $F=\eta(A_p)$ for some $p\in F$.  If $s$ is the largest SW-subsegment of $p$, then $F=\eta(A_s)=f(s)$, as desired.
\end{proof}

\begin{theorem}\label{thm_NF_CJ}
For $X\subseteq\Seg(\lambda)$, the set $X$ is a face of the nonfriendly complex if and only if there exists $F\in\GT(\lambda)$ such that $F=\bigvee_{s\in X}f(s)$ is the canonical join representation of $F$.
\end{theorem}

\begin{figure}
$$\begin{array}{ccccccccccc}\includegraphics[scale=1.9]{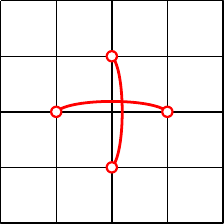} & \raisebox{.8in}{$\mapsto$} & \includegraphics[scale=1.9]{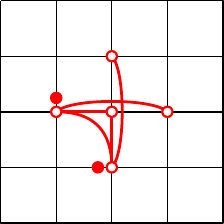} & \raisebox{.8in}{$\mapsto$} & \includegraphics[scale=.475]{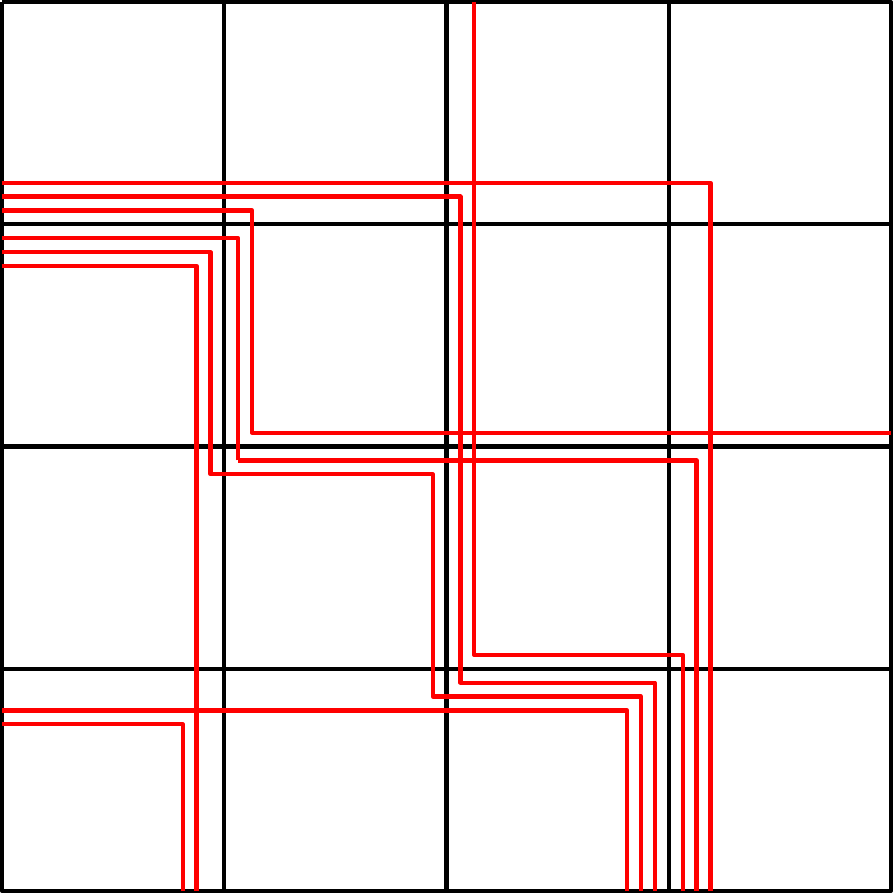}\\
X \in \Gamma^{NF}(\lambda) & \mapsto & \bigvee_{s \in X} A_s & \mapsto & \bigvee_{s \in X} f(s) \in \wtil{\Delta}^{NK}(\lambda) \end{array}$$
\caption{{A face of the nonfriendly complex and the corresponding face of the reduced nonkissing complex.}}
\label{NF_to_NK}
\end{figure}

%currently the pen width in the right figure is "fat"

\begin{proof}
Since the nonfriendly complex and the canonical join complex are both flag complexes, it suffices to prove the statement when $|X|\leq 2$.

The empty set is both a face of the nonfriendly complex and the canonical join representation of $F_0$, so the statement holds for $|X|=0$. The case $|X|=1$ was handled in Lemma~\ref{cl_join-irreducible_bijection} since the only elements in a lattice whose canonical join representation is itself are the join-irreducibles.

Let $X=\{s,t\}$, $s\neq t$. Assume $s$ and $t$ are friendly. Thus there exists a common subsegment $u$ of $s$ and $t$ along which $s$ and $t$ are friendly. We note that there may be many valid choices for $u$, which we will choose from arbitrarily.

Without loss of generality, we may assume $s$ either starts with $u$ or enters $u$ from the West, and $s$ either ends with $u$ or leaves $u$ to the South. Similarly, $t$ either starts with $u$ or enters $u$ from the North, and $t$ either ends with $u$ or leaves $u$ to the East. Divide $s$ into three segments $s=s_1\circ u\circ s_2$ where $s_1$ or $s_2$ may be an empty segment. Since $u$ is a SW-subsegment of $t$, we have $A_u\subseteq A_t$. Similarly, $s_1$ and $s_2$ are SW-subsegments of $s$, so $A_{s_i}\subseteq A_s$ for $i=1,2$. Hence,
$$A_s\vee A_t\subseteq (A_{s_1}\vee A_{s_2}\vee A_u)\vee A_t=A_{s_1}\vee A_{s_2}\vee A_t\subseteq A_s\vee A_t,$$
so we have $\eta(A_s)\vee\eta(A_t)=\eta(A_{s_1})\vee\eta(A_{s_2})\vee\eta(A_t)$. If both $s_1$ and $s_2$ are empty, then the join representation $\eta(A_s)\vee\eta(A_t)$ is redundant. Otherwise, since $\eta(A_{s_1})<\eta(A_s)$ and $\eta(A_{s_2})<\eta(A_s)$, we have found a join-refinement of $\eta(A_s)\vee\eta(A_t)$.

Now assume that $s$ and $t$ are nonfriendly. It is clear that neither segment is a SW-subsegment of the other, so $A_s\vee A_t$ is irredundant. Suppose we have another irredundant join representation so that $A_s\vee A_t=\bigvee_{i=1}^l A_{u_i}$ for some segments $u_1,\ldots,u_l$. We prove that $s\in A_{u_i}$ and $t\in A_{u_j}$ for some $i$ and $j$. This would imply that $\eta(A_s)\vee\eta(A_t)$ is a canonical join representation.

Suppose to the contrary that $s\notin A_{u_i}$ for any $i$. Then $s=u_1^{\pr}\circ\cdots\circ u_m^{\pr},\ m>1$ where each $u_i^{\pr}$ is in $A_{u_j}$ for some $j$ (depending on $i$). There exist indices $i\leq j$ such that if $u^{\pr}=u_i^{\pr}\circ\cdots\circ u_j^{\pr}$:
\begin{itemize}
\item $s\neq u^{\pr}$,
\item either $i=1$ or $s$ enters $u^{\pr}$ from the West, and
\item either $j=m$ or $s$ leaves $u^{\pr}$ to the South.
\end{itemize}
Since $u^{\pr}\in A_s\vee A_t$, we have $u^{\pr}=t_1\circ\cdots\circ t_k$ where each $t_i$ is a SW-subsegment of $s$ or $t$. By the hypotheses on $u^{\pr}$, either $t_1\in A_t$ or $t_k\in A_t$ must hold. Suppose $t_1\in A_t$. Choose $j$ maximal such that $t_1\circ\cdots\circ t_j$ is a SW-subsegment of $t$. If $j=k$, then $u^{\pr}$ is a SW-subsegment of $t$, and $s$ and $t$ are friendly along $u^{\pr}$. If $j<k$ then $t_{j+1}$ is a SW-subsegment of $s$, so $s$ leaves $t_j$ to the South while $t$ either ends at $t_j$ or leaves $t_j$ to the East. In either case, we conclude that $s$ and $t$ are friendly along $t_1\circ\cdots\circ t_j$. Since $s$ and $t$ were assumed to be nonfriendly, we have obtained a contradiction.
\end{proof}

Let $L$ be a semidistributive lattice. For any covering $x\lessdot y$ in $L$, there exists a unique join-irreducible $j$ such that $x\vee j=y$ and $x\wedge j=j_*$ \cite[Lemma 3.3]{barnard:2016canonical}. Furthermore, Barnard proved that $y=\bigvee A$ is a canonical join representation if and only if $A$ is the set of join-irreducibles $j$ such that there exists $x\lessdot y$ where $x\vee j=y$ and $x\wedge j=j_*$. We identity the join-irreducibles corresponding to covering relations in the Grid-Tamari order in the following claim.

\begin{lemma}\label{cl_join-irreducible_label}
Let $F$ and $F^{\pr}$ be adjacent facets of $\Delta^{NK}(\lambda)$.  If $F\stackrel{s}{\ra}F^{\pr}$, then $F\vee\eta(A_s)=F^{\pr}$ and $F\wedge\eta(A_s)=\eta(A_s)_*$.
\end{lemma}

\begin{proof}
If $F\stackrel{s}{\ra}F^{\pr}$, then all proper SW-subsegments of $s$ are in $\phi(F)$, but not $s$ itself. Since $A_s\setm\{s\}$ is biclosed, we have $\phi(F)\wedge A_s=A_s\setm\{s\}$. Since $A_s$ is the minimum biclosed set whose image under $\eta$ is equal to $\eta(A_s)$, we have $\eta(A_s\setm\{s\})<\eta(A_s)$. By Lemma~\ref{lem_lattice_quotient_cover}, $\eta(A_s)$ covers $\eta(A_s\setm\{s\})$. Since $\eta(A_s)$ is join-irreducible, we have $\eta(A_s\setm\{s\})=\eta(A_s)_*$. Putting this together, we have 
\begin{align*}
F\wedge\eta(A_s) &= \eta\circ\phi(F)\wedge\eta(A_s)\\
&=\eta(\phi(F)\wedge A_s)\\
&=\eta(A_s\setm\{s\})\\
&=\eta(A_s)_*.
\end{align*}

Since $F$ and $\eta(A_s)$ are incomparable, $F<F\vee\eta(A_s)$. But $A_s\subseteq\phi(F^{\pr})$, so $\eta(A_s)\leq F^{\pr}$. This implies that $F\vee\eta(A_s)\leq F^{\pr}$. As $F^{\pr}$ covers $F$, we conclude that $F\vee\eta(A_s)=F^{\pr}$, as desired.
\end{proof}

Given an element $F\in\GT(\lambda)$, we define the \emph{descent set} $\Des(F)$ to be the set of segments $s$ such that there exists a facet $F^{\pr}$ adjacent to $F$ with $F^{\pr}\stackrel{s}{\ra}F$. Dually, the \emph{ascent set} $\Asc(F)$ is the set of segments $s$ such that $F\stackrel{s}{\ra}F^{\pr}$ for some facet $F^{\pr}$. By the discussion before Lemma~\ref{cl_join-irreducible_label}, we obtain the following corollary to Theorem~\ref{thm_NF_CJ}.

\begin{corollary}\label{cor_lower_cover}
For $X\subseteq\Seg(\lambda)$, there exists $F\in\GT(\lambda)$ such that $X=\Des(F)$ if and only if $X$ is a face of the nonfriendly complex.
\end{corollary}

\subsection{Lattice congruences}\label{subsec_gt_cong}

{For this section, we partially order the set of segments by inclusion: $s\leq t$ if $s\subseteq t$.}

Since $\GT(\lambda)$ is a congruence-uniform lattice, the join-irreducibles of $\GT(\lambda)$ are in bijection with the join-irreducibles of $\Con(\GT(\lambda))$ via the map $j\mapsto\con(j_*,j)$ for $j\in \JI(\GT(\lambda))$.  Composing with $f$ defines a bijection between $\Seg(\lambda)$ and join-irreducibles of $\Con(\GT(\lambda))$. In fact, we have the following theorem.

\begin{theorem}\label{thm_congruences}
The poset $\Con(\GT(\lambda))$ is isomorphic to the lattice of order filters of $\Seg(\lambda)$.
\end{theorem}

\begin{proof}[Proof of Theorem~\ref{thm_congruences}]By Lemma \ref{cl_join-irreducible_label}, if $F\stackrel{s}{\ra}F^{\pr}$, then $F\equiv F^{\pr}\mod\con(\eta(A_s)_*,\eta(A_s))$ and $\eta(A_s)_*\equiv\eta(A_s)\mod\con(F,F^{\pr})$.

%As a consequence of Lemma \ref{cl_join-irreducible_label}, if $F\stackrel{s}{\ra}F^{\pr}$, then $F\equiv F^{\pr}\mod\con(\eta(A_s)_*,\eta(A_s))$ and $\eta(A_s)_*\equiv\eta(A_s)\mod\con(F,F^{\pr})$.

For $X\in\Bic(\lambda)$, define $X^{\downarrow s}=\ov{X^{\downarrow}-S_{\geq s}}$, where $S_{\geq s}$ is the set of segments containing $s$.  It is straight-forward to check that $X^{\downarrow s}$ is biclosed.  Moreover the relation $X\equiv Y\mod\Theta_s$ if $X^{\downarrow s}=Y^{\downarrow s}$ is a lattice congruence of $\Bic(\lambda)$ coarser than $\Theta$.  Since $\GT(\lambda)$ is isomorphic to $\Bic(\lambda)/\Theta$, this congruence decends to a lattice congruence on $\GT(\lambda)$.  From the discussion following Lemmas \ref{cl_join-irreducible_bijection} and \ref{cl_join-irreducible_label}, the congruence $\Theta_s$ contracts exactly those covering relations in $\GT(\lambda)$ labelled by a segment $t$ containing $s$.

To complete the proof of Theorem \ref{thm_congruences}, it remains to show that $\con(\eta(A_s)_*,\eta(A_s))=\Theta_s$.

\begin{lemma}\label{cl_end_contraction}
Let $s,t$ be segments such that $s$ is an initial or terminal subsegment of $t$.  Then $\eta(A_t)_*\equiv\eta(A_t)\mod\con(\eta(A_s)_*,\eta(A_s))$.
\end{lemma}

\begin{proof}
We assume $t=s\circ s^{\pr}$ for some segment $s^{\pr}$.  The case $t=s^{\pr}\circ s$ is similar.

Let $X=\ov{A_s\cup A_{s^{\pr}}}$.  Then $X$ consists of segments that can be decomposed as a terminal SW-subsegment of $s$ and an initial SW-subsegment of $s^{\pr}$.  From this observation, we deduce that the sets
$$X-\{s\},X-\{s^{\pr}\},X-\{s,t\},X-\{s^{\pr},t\},X-\{s,t,s^{\pr}\}$$
are all biclosed.  Moreover, as $X$ constains all SW-subsegments of $s,s^{\pr}$ and $t$, only one of these covering relations is contracted by $\Theta$.  That is, this hexagonal subposet of $\Bic(S)$ is mapped to a pentagonal subposet of $\GT(\lambda)$ under $\eta$.  In particular, there are covering relations $(Y,Z),(Y^{\pr},Z^{\pr})$ in $\Bic(S)$ not contracted by $\Theta$ and labelled $s,t$ respectively such that $Y^{\pr}\equiv Z^{\pr}\mod\con(Y,Z)$.

Then $\eta(Y)\stackrel{s}{\ra}\eta(Z)),\ \eta(Y^{\pr})\stackrel{t}{\ra}\eta(Z^{\pr})$ are covering relations of $\GT(\lambda)$ since $(Y,Z)$ and $(Y^{\pr},Z^{\pr})$ are not contracted by $\Theta$.  Moreover, $\eta(Y^{\pr})\equiv\eta(Z^{\pr})\mod\con(\eta(Y),\eta(Z))$.  By Lemma \ref{cl_join-irreducible_label}, we deduce that $\eta(A_t)_*\equiv\eta(A_t)\mod\con(\eta(A_s)_*,\eta(A_s))$, as desired.\end{proof}

If $s\subseteq t$, then by first extending $s$ to an initial subsegment of $t$ and applying Lemma \ref{cl_end_contraction} twice, we deduce that $\eta(A_t)_*\equiv\eta(A_t)\mod\con(\eta(A_s)_*,\eta(A_s))$.  Therefore, $\con(\eta(A_s)_*,\eta(A_s))=\Theta_s$ holds, and Theorem \ref{thm_congruences} is proved.\end{proof}

\subsection{The lattice theoretic shard intersection order}

{In this section, we classify wide sets of segments lattice-theoretically using $\Psi^l(\lambda)$. We use this this description to conclude that $\Psi^l(\lambda)$ is isomorphic to  $\Psi^w(\lambda)$. Before doing so, it will be useful to show how intervals of $\text{Bic}(\lambda)$ such as those appearing in the definition of $\Psi^l(\lambda)$ may be regarded as lattices of biclosed sets on a restricted set of segments (see Proposition~\ref{prop_faces_GT}) as follows.} 

{If $S$ is a subset of $\Seg(\lambda)$, we say that a subset $X\subseteq S$ is \emph{closed relative to $S$} if $X=\ov{X}\cap S$. It is \emph{coclosed relative to $S$} if $(S\setm X)=\ov{S\setm X}\cap S$. We define $\Bic(S)$ as the collections of segments that are closed and coclosed relative to $S$, ordered by inclusion.}

\begin{lemma}\label{meet_of_B_is}
Let $X \in \text{Bic}(\lambda)$, and let $s_1, \ldots, s_l \in X$ be segments such that each $X\backslash\{s_i\}$ is biclosed. We have that $\bigwedge_{i = 1}^l X\backslash\{s_i\} = X\backslash\overline{\{s_1,\ldots, s_l\}}$.\end{lemma}
\begin{proof}
Since each $X\backslash\{s_i\}$ is biclosed, we know that $X^c\cup \{s_i\}$ is biclosed for each $i \in [l]$. 

We claim that $X^c \cup \overline{\{s_1,\ldots, s_l\}}$ is closed. It is enough to show that given composable segments $s \in X^c$ and $s_{i_1}\circ \cdots \circ s_{i_\ell} \in \overline{\{s_{1}, \ldots, s_{l}\}}$ we have $s \circ s_{i_1}\circ \cdots \circ s_{i_\ell}  \in X^c\cup \overline{\{s_{1}, \ldots, s_l\}}.$ Since $X^c \cup \{s_{i_1}\}$ is closed, we know that $s \circ s_{i_1} \in X^c\cup \overline{\{s_{1}, \ldots, s_l\}}$. Moreover, $s \circ s_{i_1} \in X^c$ since $s \not \in \overline{\{s_1,\ldots, s_l\}}$. Since $X^c \cup \{s_{i_2}\}$ is closed, we obtain that $s \circ s_{i_1}\circ s_{i_2} \in X^c\cup \overline{\{s_1,\ldots, s_l\}}$. As above, $s \circ s_{i_1} \circ s_{i_2} \in X^c$ since $s \not \in \overline{\{s_1,\ldots, s_l\}}$. Continuing with this argument, we obtain that $X^c \cup \overline{\{s_1,\ldots, s_l\}}$ is closed.

Now Lemma~\ref{Lemma_meet_in_bic} implies the following
$$\begin{array}{rclll}
\bigwedge_{i = 1}^l X\backslash\{s_i\} 
& = & \left(\bigvee_{i=1}^l X^c \cup \{s_i\}\right)^c \\
& = & \left(\overline{\bigcup_{i=1}^l X^c \cup \{s_i\}}\right)^c \\%& \text{(using that $X^c\cup \{s_i\}$ is closed)} \\
& = & \left(\overline{X^c \cup (\bigcup_{i=1}^l \{s_i\})}\right)^c \\
& = & \left(\overline{X^c \cup \overline{\{s_1,\ldots, s_l\}}}\right)^c \\
& = & \left(X^c \cup \overline{\{s_1,\ldots, s_l\}}\right)^c & \text{(using that $X^c \cup \overline{\{s_1,\ldots, s_l\}}$ is closed)} \\
& = & X\backslash\overline{\{s_1, \ldots, s_l\}}.
\end{array}$$\end{proof}

\begin{proposition}\label{prop_faces_GT}
Let $X$ be a biclosed set of segments, and let $s_1,\ldots,s_l\in X$ such that $X\setm\{s_i\}$ is biclosed. Then $\Bic(\ov{\{s_1,\ldots,s_l\}})$ is isomorphic to the interval $[\bigwedge_{i=1}^l X\setm\{s_i\},\ X]$. Furthermore, this isomorphism descends to the quotient mod $\Theta$; that is,
$$\Bic(\ov{\{s_1,\ldots,s_l\}})/\overline{\Theta}\cong[\bigwedge_{i=1}^l X\setm\{s_i\},\ X]/\Theta$$ where $Y_1, Y_2 \in \Bic(\ov{\{s_1,\ldots,s_l\}})$ satisfy $Y_1 \equiv Y_2 \text{ mod } \overline{\Theta}$ if and only if $\overline{\eta}(Y_1) = \overline{\eta}(Y_2)$ where $\overline{\eta}$ is the lattice quotient map $\overline{\eta}: \Bic(\ov{\{s_1,\ldots,s_l\}}) \to \Bic(\ov{\{s_1,\ldots,s_l\}})/\overline{\Theta}.$
\end{proposition}

\begin{proof}
We first show that for any $s_i$ and $s_j$ there does not exist a segment $t \in \Seg(\lambda)$ such that $s_i\circ t = s_j.$ Suppose that such a segment $t$ does exist. As $s_i \not \in X\setm\{s_i\}$ and $X\setm\{s_i\}$ is biclosed, we have $t \in X\setm\{s_i\}$. This implies $t \in X$ and thus $t \in X\setm\{s_j\}$. However, this means that $s_i, t \in X\setm\{s_j\}$, but $s_j = s_i\circ t \not \in X\setm\{s_j\}$. This contradicts that $X\setm\{s_j\}$ is biclosed.

Next, we show that the map $Y^\prime \mapsto Y^\prime \cap \overline{\{s_1, \ldots, s_l\}}$ from $[\bigwedge_{i = 1}^l X\backslash\{s_i\}, X]$ to $\text{Bic}(\overline{\{s_1,\ldots, s_l\}})$ is an isomorphism of posets. It is easy to see that this map is well-defined and order-preserving. On the other hand, its inverse $Y \mapsto Y^\prime := (\bigwedge_{i=1}^l X\backslash \{s_i\}) \cup Y$ is clearly order-preserving so it remains to prove that its inverse is well-defined. To do so, we show that $Y^\prime \in \Bic(\lambda)$. 

To see that $Y^\prime$ is closed, it is enough show that if $t_1 \in \bigwedge_{i=1}^l X\backslash \{s_i\}$ and $t_2 \in Y$ are composable, then $t_1\circ t_2 \in Y^\prime$. Since $X$ is closed, we know $t_1\circ t_2 \in X$. Suppose $t_1 \circ t_2 \not \in \bigwedge_{i=1}^l X\backslash \{s_i\}$. By Lemma~\ref{meet_of_B_is}, we know that $t_1\circ t_2 = s_{i_1}\circ \cdots \circ s_{i_k}$ for some $s_{i_1}, \ldots, s_{i_k} \in \overline{\{s_1, \ldots, s_l\}}.$ Now write $t_2 = s_{j_1}\circ \cdots \circ s_{j_r}$ where $s_{j_1}, \ldots, s_{j_r} \in Y$. If $s_{i_k} \subseteq s_{j_r}$, then $s_{j_r} = s_{i_k} \circ t^\prime$ for some segment $t^\prime \in \text{Seg}(\lambda)$. However, such an equation contradicts the result from the first paragraph of the proof. The analogous argument shows $s_{j_r}$ is not properly contained in $s_{i_k}$. We conclude that $s_{j_r} = s_{i_k}$. By repeating this argument and removing pairs of equal segments $s_{j_n} = s_{i_m}$ with $n \le r$ and $m \le k$, we either obtain an equation $t_1 \circ s_{j_1}\circ \cdots \circ s_{j_{n-1}} = s_{i_1}$ or $t_1 = s_{i_1}\circ \cdots \circ s_{i_{m-1}}$. In either case, we reach a contradiction.

We now show that $Y^\prime$ is co-closed. Let $t_1, t_2 \not \in Y^\prime$ be composable. We can assume $t_1\circ t_2 \in X$, otherwise we are done. Since $X$ is co-closed, we can assume $t_2 \in X$. We also know that $\bigwedge_{i=1}^l X\backslash \{s_i\}$ is co-closed so $t_1\circ t_2 \not \in \bigwedge_{i=1}^l X\backslash \{s_i\}$. Now by Lemma~\ref{meet_of_B_is}, we have $t_1\circ t_2 = s_{i_1}\circ \cdots \circ s_{i_k}$ for some $s_{i_1},\ldots, s_{i_k} \in \overline{\{s_1,\ldots, s_l\}}$ and $t_2 = s_{j_1}\circ \cdots \circ s_{j_r}$ for some $s_{j_1}, \ldots, s_{j_r} \in \overline{\{s_1,\ldots, s_l\}}.$ Using the argument from the previous paragraph, we either obtain an equation $t_1 = s_{i_1}\circ \cdots \circ s_{i_{m-1}}$ so $t_1 \in X\backslash Y^\prime$. We conclude that $t_1\circ t_2 \in X\backslash Y^\prime$. It follows that $\Bic(\ov{\{s_1,\ldots,s_l\}})$ is isomorphic to the interval $[\bigwedge_{i=1}^l X\setm\{s_i\},\ X]$.

Lastly, we show that this isomorphism descends to the quotient mod $\Theta$. To do so, we show that for any $Y_1^\prime, Y_2^\prime \in \Bic(\lambda)$ with corresponding relatively biclosed sets $Y_1, Y_2 \in \Bic(\ov{\{s_1,\ldots,s_l\}})$, one has $\eta(Y_1^\prime) = \eta(Y_2^\prime)$ if and only if $\overline{\eta}(Y_1) = \overline{\eta}(Y_2)$. Using \cite[Claim 8.9]{mcconville:2015lattice}, we show that ${Y_1^\prime}^\downarrow = {Y_2^\prime}^\downarrow$ if and only if ${Y_1}^{\downarrow, \overline{\Theta}} = {Y_2^{\downarrow, \overline{\Theta}}}$ where ${Y_i}^{\downarrow, \overline{\Theta}} := \{s \in Y_i: \ A_s\cap \overline{\{s_1, \ldots, s_l\}} \subseteq Y_i\}$ for $i = 1, 2$. We only show that the latter implies the former as the converse is clear.

Suppose that ${Y_1}^{\downarrow, \overline{\Theta}} = {Y_2^{\downarrow, \overline{\Theta}}}$. Let $s \in {Y_1^\prime}^\downarrow$. Note that $A_s = (A_s\cap (\bigwedge_{i=1}^lX\backslash\{s_i\}))\sqcup (A_s \cap Y_1)$ so it is enough to show that any segment $t \in A_s\cap Y_1$ belongs to $Y_2$. If $t \in A_s\cap Y_1$, then $A_t \subseteq A_s$ so $t \in A_t\cap \overline{\{s_1,\ldots, s_l\}} \subset Y_1$. By assumption, $t \in A_t\cap \overline{\{s_1,\ldots, s_l\}} \subset Y_2.$ We conclude that $s \in {Y_2^\prime}^\downarrow$. The proof of the opposite inclusion is similar.\end{proof}

Let $\Psi^l(\lambda)$ be the shard intersection order of the congruence-uniform lattice $\GT(\lambda)$. Let $f:\Seg(\lambda)\ra\JI(\GT(\lambda))$ be the bijection from Lemma~\ref{cl_join-irreducible_bijection}, namely $f(s)=\eta(A_s)$. As $\GT(\lambda)$ is congruence-uniform, this extends to a bijection $\hat{f}:\Seg(\lambda)\ra\JI(\Con(\GT(\lambda)))$ where $\hat{f}(s)=\con(f(s)_*,f(s))$.

\begin{theorem}\label{thm_lattice_shard_intersection}
A set $T$ of segments is wide if and only if there exists an element $x\in\GT(\lambda)$ such that
$$\{\hat{f}(s):\ s\in T\}=\{\con(w,z):\ \bigwedge_{i=1}^l y_i\leq w\lessdot z\leq x\},$$
where $y_1,\ldots,y_l$ are the elements covered by $x$. Consequently, the posets $\Psi^{w}(\lambda)$ and $\Psi^l(\lambda)$ are isomorphic.
\end{theorem}

\begin{proof}
Let $x\in\GT(\lambda)$, and set
$$T=\{\hat{f}^{-1}(\con(w,z)):\ \bigwedge_{i=1}^l y_i\leq w\lessdot z\leq x\},$$
where $y_1,\ldots,y_l$ are the elements of $\GT(\lambda)$ covered by $x$. Then, we have
$$T=\{s:\ \exists w\stackrel{s}{\ra}z,\ \bigwedge_{i=1}^l y_i\leq w\lessdot z\leq x\}.$$
Let $X=\phi(x)$, and let $s_i$ be the segment labeling $y_i\stackrel{s_i}{\ra}x$. Then $\{s_1,\ldots,s_l\}$ is a face of the nonfriendly complex. Since $\eta(X\setm\{s_i\})=y_i$, the interval $[\bigwedge_{i=1}^l X\setm\{s_i\},X]$ maps to $[\bigwedge_{i=1}^l y_i,x]$ under $\eta$. By Proposition~\ref{prop_faces_GT}, the interval $[\bigwedge_{i=1}^l y_i,x]$ is isomorphic to $\Bic(\ov{\{s_1,\ldots,s_l\}})/{\overline{\Theta}}$, and this isomorphism preserves the edge labels. Hence, $T$ is equal to $\ov{\{s_1,\ldots,s_l\}}$. By Proposition~\ref{prop_nonfriendly_transitive}, this set is wide.

Next, we prove the converse statement. Let $T$ be a wide set and $X=\NF(T)$. Then $\bigvee_{s\in X}\eta(A_s)$ is the descent set of some element $x\in\GT(\lambda)$. Hence,
$$X=\Des(x)=\{\hat{f}^{-1}(\con(y,x)):\ y\lessdot x\}.$$
By Proposition~\ref{prop_nonfriendly_transitive}, we have $T=\ov{X}$. By the previous argument,
$$\ov{X}=\{\hat{f}^{-1}(\con(w,z)):\ \bigwedge_{i=1}^ly_i\leq w\lessdot z\leq x\}$$
where $y_1,\ldots,y_l$ are the elements of $\GT(\lambda)$ covered by $x$.
\end{proof}

\section{The Grid-associahedron fan}\label{sec_fan}

In this section, we realize the reduced nonkissing complex $\wtil{\Delta}^{NK}(\lambda)$ as the faces of a complete simplicial fan $\Fcal_{\lambda}$. Then we prove that the Grid-Tamari order is the partial order on maximal cones of $\Fcal_{\lambda}$ induced by a certain linear functional. This fan realization induces an alternate lattice known as the (geometric) shard intersection order $\Psi^f(\lambda)$. We prove that this version of the shard intersection order is isomorphic to the poset of wide sets in Section~\ref{sec_shard}.

\subsection{Fans}\label{subsec_fan_fans}

We recall some basic definitions for polyhedral fans. Fix a real vector space $V=\Rbb^r$. A polyhedral cone $C$ is \emph{pointed} if $\{0\}$ is a face of $C$. A pointed cone is minimally generated by a unique set of vectors, which are called the \emph{extreme rays}. A pointed cone is \emph{simplicial} if its extreme rays are linearly independent.

A \emph{fan} is a finite set of cones $\Fcal=\{C_1,\ldots,C_N\}$ in $V$ such that
\begin{itemize}
\item every face of $C_i$ is in $\Fcal$ for all $i\in[N]$, and
\item $C_i\cap C_j$ is in $\Fcal$ for all $i,j\in[N]$.
\end{itemize}
We will consider fans that are \emph{complete}, \emph{pointed}, and \emph{simplicial}. This means that $V=\bigcup_{i=1}^N C_i$, every cone is pointed, and every cone is simplicial, respectively. The following lemma gives a well-known characterization of simplicial fans.

\begin{lemma}\label{lem_comp_simp_fan}
Let $\Fcal$ be a set of pointed cones $\{C_1,\ldots,C_N\}$ such that every face of a cone in $\Fcal$ is also in $\Fcal$. Then $\Fcal$ is a complete simplicial fan if and only if for all $x\in V$, there exists a unique cone $C_i$ such that $x$ is a (strictly) positive linear combination of the extreme rays of $C_i$, and this linear combination is unique.
\end{lemma}

\begin{proof}
A point $x$ is a (strictly) positive linear combination of the extreme rays of $C_i$ exactly when it lies in the relative interior of $C_i$.

If $\Fcal$ is a complete fan, then $V$ is the disjoint union of the relative interiors of the cones in $\Fcal$. If a cone $C_i$ is simplicial, then each point in its relative interior may be uniquely expressed as a positive linear combination of the extreme rays of $C_i$.

Conversely, suppose for all $x\in V$, there is a unique cone $C_i$ for which $x$ is a positive linear combination of the extreme rays of $C_i$, and this linear combination is unique. The uniqueness of the linear combination implies that every cone is simplicial. Completeness follows from the existence of a cone containing any given point $x$. It remains to show that the intersection $C_{ij}=C_i\cap C_j$ of two distinct cones $C_i$ and $C_j$ is a face of each. Since the relative interiors of $C_i$ and $C_j$ are disjoint, the intersection $C_{ij}$ lies on the boundary of each. Hence, there exists a minimal face $F_i$ of $C_i$ and $F_j$ of $C_j$ each containing $C_{ij}$. A point in the relative interior of $C_{ij}$ is also in the relative interior of $F_i$ and $F_j$, which implies $F_i=F_j$ since they are cones in $\Fcal$. We conclude that $C_{ij}=F_i=F_j$ is a face of $C_i$ and $C_j$.
\end{proof}

\subsection{Fan realization of the nonkissing complex}\label{subsec_fan_nonkissing}

Fix a shape $\lambda$, and let $V^o$ be the set of interior vertices of $\lambda$. For a boundary path $p$, let $g_p$ be the vector in $\Rbb^{V^o}$ such that for $v\in V^o$,
$$g_p(v)=\begin{cases}1\ &\mbox{if }p\mbox{ enters }v\mbox{ from the North and leaves to the East,}\\-1\ &\mbox{if }p\mbox{ enters }v\mbox{ from the West and leaves to the South,}\\0\ &\mbox{otherwise.}\end{cases}$$

Given a face $F$ in $\wtil{\Delta}^{NK}(\lambda)$, let $C(F)$ be the cone generated by $\{g_p:\ p\in F\}$. Let $\Fcal_{\lambda}=\{C(F):\ F\in\wtil{\Delta}^{NK}(\lambda)\}$. We refer to $\Fcal_{\lambda}$ as the \emph{Grid-associahedron fan}. This terminology is justified by the following theorem.

\begin{theorem}\label{thm_fan}
The set $\Fcal_{\lambda}$ is a complete simplicial fan.
\end{theorem}

\begin{proof}
We prove by induction on $|V^o|$ that for each point $x\in\Rbb^{V^o}$ there exists a unique $F\in\wtil{\Delta}^{NK}(\lambda)$ such that $x$ is a (strictly) positive linear combination of vectors in $\{g_p:\ p\in F\}$, and this expression for $x$ is unique.

Fix a point $x\in\Rbb^{V^o}$. Let $w$ be an interior vertex such that the two vertices immediately South and East of $w$ are boundary vertices. Replacing $\lambda$ by its transpose $\lambda^t$ if necessary, we may assume that $x(w)\geq 0$. Let $\lambda^{\pr}$ be the same shape as $\lambda$ except that $w$ is treated as a boundary vertex of $\lambda^{\pr}$, and let $x^{\pr}$ be the restriction of $x$ to $V^o\setm\{w\}$. By the inductive hypothesis, there exists a unique face $F^{\pr}$ of $\wtil{\Delta}^{NK}(\lambda^{\pr})$ such that $x^{\pr}$ is a positive linear combination of vectors in $\{g_p:\ p\in F^{\pr}\}$. Moreover, the coefficients $a_p>0$ in the equation $x^{\pr}=\sum_{p\in F^{\pr}}a_pg_p$ are unique.

We construct a face $F$ with $x\in C(F)$ by extending elements of $F^{\pr}$ as follows. A path in $\lambda^{\pr}$ that ends at $w$ may be extended by one step South or East to form a boundary path of $\lambda$. If $p\in F^{\pr}$ does not end at $w$, then the path is left alone. If $x(w)\geq 0$ and $p\in F^{\pr}$ enters $w$ from the West, then $p$ is extended to the East. Similarly, if $x(w)\leq 0$ and $p\in F^{\pr}$ enters $w$ from the North, then $p$ is extended to the South.

The remaining paths are extended by the following rule. Suppose $x(w)>0$, and let $e$ be the vertical edge in $\lambda$ with southern endpoint $w$. Let $\{p_1,\ldots,p_k\},\ (k\geq 0)$ be the set of paths in $F^{\pr}$ containing $e$, ordered such that $p_k\prec_e\cdots\prec_e p_1$. Let $p_{k+1}$ be the vertical path in $\lambda^{\pr}$ containing $e$, which implies $p_{k+1}\prec_e p_k$ by definition. We note that $x^{\pr}=\sum_{p\in F^{\pr}\cup\{p_{k+1}\}}a_pg_p$ for any value of $a_{p_{k+1}}$ since $g_{p_{k+1}}=0$. For convenience, we set $a_{p_{k+1}}=x(w)+1$. There exists a unique index $l\in[1,k+1]$ such that
$$\sum_{i=1}^{l-1}a_{p_i}< x(w)\leq\sum_{i=1}^la_{p_i}.$$

We extend each of the paths $p_1,\ldots,p_l$ one step East, and each of the paths $p_{l+1},\ldots,p_k$ one step South. Finally, if $x(w)<\sum_{i=1}^la_{p_i}$ and $l<k+1$, we also add the extension of $p_l$ one step South. The set $F$ of paths obtained by extending paths in $F^{\pr}$ in this manner is a nonkissing collection. We note that $|F|=|F^{\pr}|+1$ unless $x(w)=\sum_{i=1}^la_{p_i}$, in which case $|F|=|F^{\pr}|$. For $p\in F$, let $a_p=a_{p^{\pr}}$ if $p^{\pr}\neq p_l$ where $p^{\pr}$ is the restriction of $p$ to $V^o(\lambda^{\pr})$. If $p$ is the extension of $p_l$ one step East, we set $a_p=x(w)-\sum_{i=1}^{l-1}a_{p_i}$. If $x(w)<\sum_{i=1}^la_{p_i}$ and $p$ is the extension of $p_l$ one step South, we let $a_p=\sum_{i=1}^la_{p_i}-x(w)$. With these coefficients, we have $x=\sum_{p\in F}a_pg_p$.

We next prove that this expression for $x$ is unique. Let $x=\sum_{p\in F}b_pg_p$ for some coefficients $b_p>0$. For a path $p\in F$, let $p^{\pr}$ be the restriction of $p$ to $\lambda^{\pr}$. Then $x^{\pr}=\sum_{p\in F}b_pg_{p^{\pr}}$. The uniqueness of this expression forces $b_p=a_{p^{\pr}}$ for all paths $p\in F$ except $p^{\pr}=p_l$ if $|F|=|F^{\pr}|+1$. For the two paths whose restriction to $\lambda^{\pr}$ is $p_l$, only one of them turns at $w$. Hence, the remaining two coefficients are uniquely determined.

To complete the proof, we show that if $x$ is in the relative interior of $C(G)$ for some face $G$ then $G=F$. Let $x\in C(G)$ such that $x=\sum_{p\in G}c_pg_p$ for some coefficients $c_p>0$. Then $x^{\pr}=\sum_{p\in G}c_pg_{p^{\pr}}$. Then $x^{\pr}$ is in the relative interior of $C(G^{\pr})$ where $G^{\pr}$ is the restriction of $G$ to $\lambda^{\pr}$. Hence, $G^{\pr}=F^{\pr}$, so $G$ is obtained by extending paths in $F^{\pr}$. Since $G$ is a nonkissing collection, it does not contain two paths $p,q$ such that $p$ enters $w$ from the West and leaves to the South while $q$ enters $w$ from the North and leaves to the East. Since $x(w)\geq 0$, this means every path that enters $w$ from the West must leave to the East.

It remains to show that there is a unique extension of the paths entering $w$ from the North. Indeed, the only possibilities to extend the paths $p_1,\ldots,p_k$ to form a nonkissing collection are to either
\begin{enumerate}
\item extend $p_1,\ldots,p_j$ East and $p_{j+1},\ldots,p_k$ South for some $j$, or
\item extend $p_1,\ldots,p_j$ East and $p_j,\ldots,p_k$ South for some $j$.
\end{enumerate}
The appropriate choice is determined by $x(w)$ and the coefficients $a_{p_i}$, as described above.
\end{proof}

As a complete simplicial fan, the intersection of any two cones in $\Fcal_{\lambda}$ is a face of each. In particular, for $g_p\neq 0$, the cone $C(F)$ contains $g_p$ if and only if $p$ is in $F$. Consequently, the identity $C(F\cap F^{\pr})=C(F)\cap C(F^{\pr})$ holds for any two faces $F,F^{\pr}$ of the reduced nonkissing complex. Thus, we have proved that $\Fcal_{\lambda}$ realizes the reduced nonkissing complex.

\subsection{Fan posets}\label{subsec_fan_poset}

Let $\Fcal$ be a complete fan in $\Rbb^n$, and let $P$ be a poset on the maximal cones of $\Fcal$. The pair $(\Fcal,P)$ is a \emph{fan poset} \cite{reading:lattice_congruence} if
\begin{itemize}
\item for all closed intervals $I\subseteq P$, the set $\bigcup_{C\in I}C$ is a polyhedral cone, and
\item for $C\in\Fcal$, the set of maximal cones in $\Fcal$ containing $C$ is a closed interval of $P$.
\end{itemize}

Our main result in this section is that the fan $\Fcal_{\lambda}$ defined in Section~\ref{subsec_fan_nonkissing} defines a fan poset.

\begin{remark}
Fan posets appear in the study of hyperplane arrangements and Coxeter combinatorics. For example, let $\mathcal{A}$ be a central hyperplane arrangement and let $\mathcal{F}$ be the complete fan consisting of the chambers of $\mathcal{A}$ and all faces of these chambers. Now let $\mathcal{P}(\mathcal{A}, B)$ be the poset of chambers of $\mathcal{A}$ with respect to a choice of base chamber $B$ of $\mathcal{A}$. The elements of $\mathcal{P}(\mathcal{A}, B)$ are the chamber of $\mathcal{A}$ where two chambers $R_1, R_2 \in \mathcal{P}(\mathcal{A}, B)$ satisfy $R_1 \leq R_2$ if $S(R_1) \subseteq S(R_2)$ where $S(R_i)$ is the set of hyperplanes of $\mathcal{A}$ separating $R_i$ from $B$. By \cite[Theorem 4.2.(i)]{reading:lattice_congruence}, $(\mathcal{F}, \mathcal{P}(\mathcal{A}, B))$ is a fan poset.

As another example, suppose $W$ is a finite Coxeter group and $c$ is choice of Coxeter element of $W$. Regard $W$ as a lattice whose partial order is the weak order. Let $\mathcal{F}_c$ denote the associated \emph{c-Cambrian fan} in the sense of \cite{reading.speyer:2009fans}, and let $W/\Theta_c$ be corresponding \emph{c-Cambrian lattice}. The data $(\mathcal{F}_c, W/\Theta_c)$ is a fan poset.\end{remark}

\begin{theorem}\label{thm_fan_order}
The pair $(\Fcal_{\lambda},\GT(\lambda))$ is a fan poset.
\end{theorem}

We divide the proof into two parts as Lemma~\ref{lem_fan_poset_star} and Lemma~\ref{lem_fan_poset_cone}. The proofs of these statements rely on the following lemma proved in \cite[Claim 8.2]{mcconville:2015lattice}.

\begin{lemma}\label{lem_AK}
Let $X\in\Bic(\lambda)$. For $p\in\eta(X)$, we have $A_p\subseteq X$ and $K_p\cap X=\emptyset$.
\end{lemma}

\begin{lemma}\label{lem_fan_poset_star}
For $F\in\wtil{\Delta}^{NK}(\lambda)$, the set of facets containing $F$ is an interval of $\GT(\lambda)$.
\end{lemma}

\begin{proof}
Let $F$ be a face of $\wtil{\Delta}^{NK}(\lambda)$. Let $\stars(F)$ be the set of facets containing $F$. We claim that $\stars(F)$ is a closed interval of $\GT(\lambda)$. Since $\wtil{\Delta}^{NK}(\lambda)$ is a flag complex,
$$\stars(F)=\bigcap_{p\in F}\stars(\{p\}).$$
In a finite lattice, the intersection of a collection of closed intervals is itself a closed interval. Hence, it suffices to show that $\stars(\{p\})$ is an interval.

Let $A_p$ and $K_p$ be SW-subsegments and NE-subsegments of $p$, respectively. Let $I=[A_p,\Seg(\lambda)\setm K_p]$ be an interval in $\Bic(\lambda)$. We show that $\eta(I)=\stars(\{p\})$, which implies $\stars(\{p\})$ is isomorphic to the quotient interval $[[A_p]_{\Theta},[\Seg(\lambda)\setm K_p]_{\Theta}]$.

Let $F^{\pr}\in\stars(\{p\})$. Since $\eta\circ\phi(F^{\pr})=F^{\pr}$, we have $A_p\subseteq\phi(F^{\pr})$ and $K_p\cap\phi(F^{\pr})=\emptyset$ by Lemma~\ref{lem_AK}. Hence, $\phi(F^{\pr})\in I$, and $F^{\pr}\in\eta(I)$.

Now let $X\in I$ be given. Let $e$ be a vertical edge of $\lambda$, and let $p_e$ be the top path at $e$ in $\eta(X)$. We show that $p$ is nonkissing with $p_e$. Since $\eta(X)$ is a maximal nonkissing collection, this would imply that $p\in\eta(X)$.

Suppose to the contrary that $p$ and $p_e$ kiss along a common segment $s$. Then $s$ is either a SW-subsegment of $p$ and a NE-subsegment of $p_e$ or vice versa. If $s$ is a SW-subsegment of $p$, then $s\in X$ since $A_p\subseteq X$. But $s\in K_{p_e}$ implies $s\notin X$ by Lemma~\ref{lem_AK}, a contradiction. Similarly, if $s$ is a NE-subsegment of $p$, then $s\notin X$ since $X\subseteq\Seg(\lambda)\setm K_p$. However, $s\in A_{p_e}$ implies $s\in X$ by Lemma~\ref{lem_AK}, a contradiction.

We have now established that $\eta(I)=\stars(\{p\})$, as desired.
\end{proof}

Using Theorem~\ref{thm_fan} with Lemma~\ref{lem_fan_poset_star}, for any face $F$ of the nonkissing complex, the subposet of $\GT(\lambda)$ on the facets of $\wtil{\Delta}^{NK}(\lambda)$ whose maximal cones in $\Fcal_{\lambda}$ contain $C(F)$ is a closed interval.

\begin{lemma}\label{lem_fan_poset_cone}
For any closed interval $I$ in $\GT(\lambda)$, the set $\bigcup_{F\in I}C(F)$ is a polyhedral cone.
\end{lemma}

\begin{proof}

Let $F_1,F_2\in\GT(\lambda)$ such that $F_1\leq F_2$. Let $C$ be the set of $x\in\Rbb^{V^o}$ such that
\begin{itemize}
\item for $s\in\Des(F_1)$, if $t$ is a SW-subsegment of $s$ then $\alpha_t(x)\geq 0$, and
\item for $s\in\Asc(F_2)$, if $t$ is a NE-subsegment of $s$ then $\alpha_t(x)\leq 0$.
\end{itemize}

We prove that $C=\bigcup_{F\in[F_1,F_2]}C(F)$ holds. Let $F\in[F_1,F_2]$, and let $p\in F$. Let $s\in\Des(F_1)$, and let $t$ be a SW-subsegment of $s$. We claim that $\alpha_t(g_p)\geq 0$.

Assume to the contrary that $\alpha_t(g_p)<0$. Let $u$ be a maximal subsegment of $t$ contained in $p$ such that $p$ enters $u$ from the West and leaves to the South. Let $e$ be the vertical edge whose South endpoint is the initial vertex in $u$. Let $p_e$ be the top path in $F$ at $e$. We prove that $p$ and $p_e$ are kissing, contrary to the assumption that $p$ and $p_e$ are in a nonkissing collection $F$.

To prove that $p$ and $p_e$ are kissing, we recall that $F=\eta\circ\phi(F)$ and $\phi(F_1)\subseteq\phi(F)$. Since $u$ is a SW-subsegment of $t$, which is a SW-subsegment of $s$, we have
$$A_u\subseteq A_t\subseteq A_s\subseteq\phi(F_1).$$
In particular, every SW-subsegment of $u$ is contained in $\phi(F)$.

Let $v$ be the initial vertex of $u$. Recall that $\eta(\phi(F))$ defines the path $p_e$ where for any interior vertex $v^{\pr}$ weakly South-West of $v$, the path $p_e$ leaves the segment $p_e[v,v^{\pr}]$ to the East if $p_e[v,v^{\pr}]$ is in $\phi(F)$, and it leaves to the South otherwise. Since every SW-subsegement of $u$ is in $\phi(F)$, the path $p_e$ must leave $u$ to the East. But this means that $p$ and $p_e$ are kissing along an initial subsegment of $u$.

By applying the same argument to $\lambda^{\tr}$, we may conclude that $C$ contains $\bigcup_{F\in[F_1,F_2]}C(F)$.

Now let $x$ be an interior point of $C$, and let $F$ be any facet of $\wtil{\Delta}^{NK}(\lambda)$ such that $C(F)$ contains $x$. We prove that $F_1\leq F$, and deduce $F\leq F_2$ by duality.

We prove that every segment $s$ in $\phi(F_1)$ is in $\phi(F)$ by induction on the length of $s$. Let $s\in\phi(F_1)$. Since $\phi(F_1)$ is the minimum element in its $\Theta$-equivalence class, every SW-subsegment of $s$ is in $\phi(F_1)$.

There exists a decomposition of $s$ into subsegments $s=s_1\circ\cdots\circ s_l$ such that each $s_i$ is a SW-subsegment of some $s^{\pr}\in\Des(F_1)$. If $l>1$, then $\phi(F)$ contains each $s_i$ by the induction hypothesis. Since $\phi(F)$ is closed, this means $s\in\phi(F)$.

Now suppose $l=1$, and let $s^{\pr}\in\Des(F_1)$ be a segment such that $s\in A_{s^{\pr}}$. Then $\alpha_s(x)>0$ since $x$ is an interior point of $C$. Hence, there exists a path $p\in F$ such that $\alpha_s(g_p)>0$. Let $t$ be a maximal subsegment of $s$ contained in $p$ such that $p$ enters $t$ from the North and leaves to the East. Then $t$ is a SW-subsegment of $p$, so $t\in\phi(F)$. Since $t$ is a NE-subsegment of $s$, we have $s=t_1\circ t\circ t_2$ where $t_1$ and $t_2$ are (possibly empty) SW-subsegments of $s$. Hence, $t_1,t_2\in\phi(F)$ by induction. Since $\phi(F)$ is closed, $s\in\phi(F)$ holds.

We have shown that $\phi(F_1)\subseteq\phi(F)$, so $F_1\leq F$. By applying the same proof to $\lambda^{\tr}$, we get $F\leq F_2$. Hence, the interior of $C$ is covered by the union of $C(F)$ for $F\in[F_1,F_2]$. Since $C$ is the closure of its interior and $\bigcup_{F\in[F_1,F_2]}C(F)$ is a closed set contained in $C$, they must be equal.
\end{proof}

The proof of Theorem~\ref{thm_fan_order} is now complete.

\subsection{Shards in the Grid-associahdron fan}

In this section, we determine the ridges of the Grid-associahedron fan. These ridges may be grouped together into codimension 1 cones called \emph{shards}, which are in bijection with segments.

We identify $V^o$ with the set of elementary basis vectors of $\Rbb^{V^o}$. For an interior vertex $v$, let $\alpha_v$ be the linear functional on $\Rbb^{V^o}$ such that $\alpha_v(v)=1$ and $\alpha_v(u)=0$ for $u\neq v$. For each segment $s\in S$, let $\alpha_s$ be the linear functional
$$\alpha_s=\sum_{v\in s}\alpha_v.$$ We use these linear functionals to identify the walls of the cones $C(F)$ in the following proposition.% (see Figure~\ref{walls_ex}).

\begin{figure}
$$\includegraphics[scale=1]{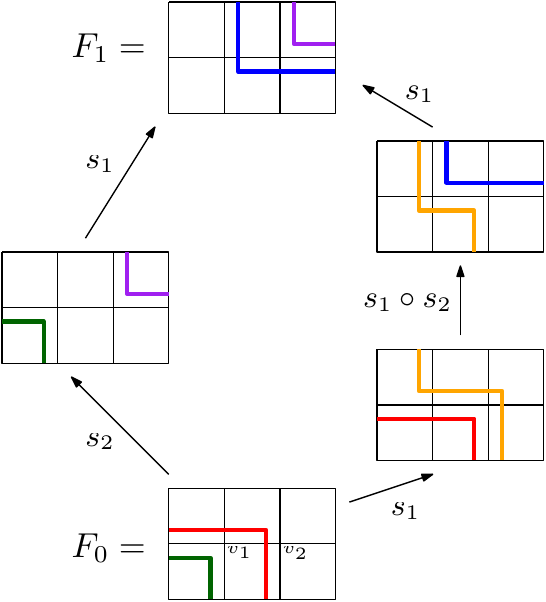} \ \ \ \ \ \raisebox{.18in}{\includegraphics[scale=2]{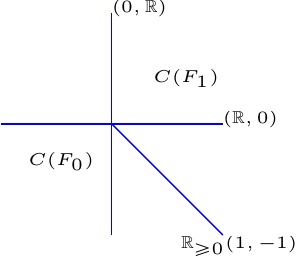}}$$
\caption{{The Grid-Tamari order $\GT(\lambda)$ and its Grid-associahedron fan $\mathcal{F}_\lambda$ when $\lambda$ is a $2\times 3$ rectangle. In $\lambda$, $s_1 = (v_1)$ and $s_2 = (v_2)$. In $\mathcal{F}_\lambda$, the horizontal (resp., vertical) hyperplane is the $v_1$-axis (resp., $v_2$-axis).}}
\label{walls_ex}
\end{figure}

\begin{proposition}\label{prop_ineq_fan}
For a facet $F\in\wtil{\Delta}^{NK}(\lambda)$, the cone $C(F)$ is defined by the inequalities
\begin{align*}
\alpha_s(x) &\leq 0\ \mbox{ if }s\in\Asc(F), \mbox{ and}\\
\alpha_s(x) &\geq 0\ \mbox{ if }s\in\Des(F).
\end{align*} 
Consequently, the map $X \mapsto C(F)$ sending face $X \in \Gamma^{NF}(\lambda)$ to the cone $C(F) \in \mathcal{F}_\lambda$ where $X = \Des(F)$ is a bijection.
\end{proposition}

\begin{proof}
Let $F^{\pr}$ be a facet adjacent to $F$, and let $s\in\Seg(\lambda)$ such that $F\stackrel{s}{\ra}F^{\pr}$. Let $F^{\pr}=F\setm\{p\}\cup\{p^{\pr}\}$ for some boundary paths $p,p^{\pr}$. The paths $p$ and $p^{\pr}$ kiss along $s$ by the definition of the edge labeling. Furthermore, $p$ enters $s$ from the West and leaves to the South, whereas $p^{\pr}$ enters $s$ from the North and leaves to the East.

Set $R=F\setm\{p\}$. We show that $\alpha_s(x)=0$ for $x\in C(R)$. Let $q\in F\setm\{p\}$. We prove that $\alpha_s(g_q)=0$, which will imply that $\alpha_s$ is $0$ on all of $C(R)$.

If $q$ enters $s$ from the North (not necessarily at the beginning of $s$), then it must leave $s$ to the South, as otherwise it would kiss $p$. If $q$ enters $s$ from the West then it must leave $s$ to the East, as otherwise it would kiss $p^{\pr}$. In either case, $q$ must turn an even number of times at vertices of $s$, alternating between positive and negative turns. Hence, $\alpha_s(g_q)=0$.

Since $p$ enters $s$ from the West and leaves to the South, it has one more negative than positive turn, so $\alpha_s(g_p)=-1$. Consequently, $\alpha_s(x)\leq 0$ for $x\in F$.

By a similar argument, if $F^{\pr}$ is adjacent to $F$ and $s$ is a segment with $F^{\pr}\stackrel{s}{\ra}F$, then $\alpha_s(x)\geq 0$ for $x\in F$. As we have found inequalities for all of the ridges of $F$, the listed inequalities suffice to define the cone $F$.
\end{proof}

For a segment $s$, let $H_s$ be the hyperplane defined by $\alpha_s(x)=0$. Let $\Delta_s$ be the set of faces $F$ of $\wtil{\Delta}^{NK}(\lambda)$ such that for all $p\in F$, if $t$ is a maximal subsegment of $s$ contained in $p$, then $p$ turns at an even number of vertices in $t$. Since $p$ alternates between positive and negative turns along the segment $t$, we have $\alpha_t(g_p)=0$. By summing over all subsegments of $s$ contained in $p$, we have $\alpha_s(g_p)=0$ as well. In particular, $C(F)\subseteq H_s$ for all $F\in\Delta_s$.

Let $\Sigma(s)=\bigcup_{F\in\Delta_s}C(F)$. We prove that $\Sigma(s)$ is a cone of codimension 1 in the following proposition.

\begin{proposition}\label{prop_shards}
The set $\Sigma(s)$ is the cone supported by $H_s$ defined by the inequalities:
\begin{itemize}
\item $\alpha_t(x)\geq 0$ if $t$ is a SW-subsegment of $s$, and
\item $\alpha_t(x)\leq 0$ if $t$ is a NE-subsegment of $s$.
\end{itemize}
\end{proposition}

\begin{proof}
We first prove that if $x\in \Sigma(s)$, then $x$ satisfies the desired inequalities.

Let $F\in\Delta_s$ and let $p\in F$ be given. Let $t$ be a SW-subsegment of $s$. We claim that $\alpha_t(g_p)\geq 0$ holds. If not, then there exists a subsegment $u$ of $t$ contained in $p$ such that $p$ enters $u$ from the West and leaves to the South. We choose $u$ to be a maximal such subsegment. Then $u$ is a SW-subsegment of $t$. But then $p$ turns at an odd numbers of vertices of $u$, which contradicts the assumption that $F\in\Delta_s$. By a similar argument, if $t$ is a NE-subsegment of $s$, then $\alpha_t(g_p)\leq 0$.

Now let $x\in\Rbb^n$ such that $\alpha_t(x)\geq 0$ if $t$ is a SW-subsegment of $s$ and $\alpha_t(x)\leq 0$ if $t$ is a NE-subsegment of $s$. Since $s$ is both a SW-subsegment and a NE-subsegment of itself, this forces $x\in H_s$. Let $F$ be the minimum face of $\wtil{\Delta}^{NK}(\lambda)$ such that $x\in C(F)$. We prove that $F\in\Delta_s$.

Let $p\in F$. If $p$ does not share any vertices with $s$, then we are done. Otherwise, let $t$ be a maximal subsegment of $s$ contained in $p$. Suppose $p$ turns at an odd number of vertices of $t$. Then either $p$ enters $t$ from the West and leaves to the South, or $p$ enters $t$ from the North and leaves to the East.

Suppose $p$ enters $t$ from the West and leaves to the South. Then $t$ is a SW-subsegment of $s$ and $\alpha_t(g_p)=-1$. Since $x\in C(F)$ and $\alpha_t(x)\geq 0$, there exists $q\in F$ such that $\alpha_t(g_q)>0$. In particular, there exists a NE-subsegment $u$ of $t$ contained in $q$ such that $q$ enters $u$ from the North and leaves to the East. But this implies $p$ and $q$ are nonkissing, contrary to the assumption that $F$ is nonkissing.

A similar contradiction may be reached under the assumption that $p$ enters $t$ from the North and leaves to the East. Therefore, $F\in\Delta_s$.\end{proof}

\begin{proposition}\label{prop_ridges_fan}
If $s$ is a segment and $F,F^{\pr}$ are adjacent facets of $\wtil{\Delta}^{NK}(\lambda)$ such that $F\stackrel{s}{\ra}F^{\pr}$, then $C(F\cap F^{\pr})\subseteq \Sigma(s)$.
\end{proposition}

\begin{proof}
Let $R=F\cap F^{\pr}$ and $p,p^{\pr}$ be paths such that $F=R\cup\{p\}$ and $F^{\pr}=R\cup\{p^{\pr}\}$. We prove that $R\in\Delta_s$.

Let $q\in R$. Suppose $q$ and $s$ have a nonempty intersection. Let $t$ be a maximal subsegment of $s$ such that $t$ is in $q$. Assume that $q$ turns at an odd number of vertices in $t$. If $q$ enters $t$ from the West and leaves to the South, then $q$ kisses $p^{\pr}$, which contradicts the assumption that $F^{\pr}$ is a nonkissing collection of paths. Similarly, if $q$ enters $t$ from the North and leaves to the East, then $q$ kisses $p$, which contradicts the assumption that $F$ is a nonkissing collection of paths. Hence, $q$ turns at an even number of vertices in $t$, and we deduce that $R\in\Delta_s$ holds.
\end{proof}

Proposition~\ref{prop_shards} and Proposition~\ref{prop_ridges_fan} imply that
$$\Sigma(s)=\bigcup\{C(R):\ R\ \mbox{is a ridge of }\wtil{\Delta}^{NK}(\lambda),\ C(R)\subseteq H_s\}.$$
We remark that there may be other faces of lower dimension supported by $H_s$ that are not in $\Sigma(s)$. If the shape $\lambda$ is sufficiently large, then a boundary path $p$ may intersect a segment $s$ in several places. For such a path, it is possible for $\alpha_s(g_p)$ to vanish, yet $\{p\}\notin\Delta_s$. However, any ridge supported by $H_s$ is in $\Sigma(s)$.

We call the cone $\Sigma(s)$ a \emph{shard}. Shards were originally introduced by Reading \cite{reading:lattice} in the following way. A \emph{real hyperplane arrangement} $\Acal$ is a finite set of hyperplanes in $\Rbb^n$. The arrangement $\Acal$ defines a complete fan on $\Rbb^n$ whose maximal faces are called \emph{chambers}. The arrangement is \emph{simplicial} if every chamber is a simplicial cone. Simplicial arrangements are exceedingly rare but include some signficant examples such as the reflection arrangement of a finite reflection group and the crystallographic arrangements as defined in {\cite{cuntz2011crystallographic}}.

Let $\Acal$ be a simplicial arrangement and $c_0$ a fixed chamber. Given a subspace $X$, let $\Acal_X$ be the set of $H\in\Acal$ such that $X\subseteq H$, and let $(c_0)_X$ be the unique chamber of $\Acal_X$ containing $c_0$. If $X=H\cap H^{\pr}$ for some $H^{\pr}\in\Acal\setm\{H\}$, then $H$ supports two faces of $\Acal_X$, which we call $H_X^+$ and $H_X^-$. Reading defined a \emph{shard} $C$ to be an inclusion-maximal cone supported by some $H\in\Acal$ such that for any $H^{\pr}\in\Acal\setm\{H\}$: if $X=H\cap H^{\pr}$ and $(c_0)_X$ is not incident to $H$, then $C$ is either supported by $H_X^+$ or $H_X^-$.

The set of chambers of $\Acal$ are partially ordered such that $c\leq c^{\pr}$ if every hyperplane $H\in\Acal$ separating $c_0$ and $c$ also separates $c_0$ and $c^{\pr}$. If $\Acal$ is simplicial, then this poset is a lattice \cite{bjorner.edelman.ziegler:lattice}. Under some additional hypotheses, the shards of $\Acal$ are in bijection with the join-irreducible lattice congruences of the poset of chambers \cite{reading:lattice_congruence}. \emph{Cambrian fans}, which realize the cluster complex, are constructed by deleting some of the shards of the reflection arrangement of a finite reflection group \cite{reading.speyer:2009fans}.

One obstruction to constructing our fans $\Fcal_{\lambda}$ in this manner is that the arrangement $\Acal_{\lambda}=\{H_s:\ s\in\Seg(\lambda)\}$ is \emph{not} simplicial in general. Furthermore, the poset of chambers of $\Acal_{\lambda}$ may not be a lattice. However, the poset of \emph{biclosed subsets} of $\Acal_{\lambda}$ does form a lattice. This is the lattice $\Bic(\lambda)$ from Section~\ref{subsec_gt_biclosed}. For a simplicial arrangement $\Acal$, the set of chambers are in bijection with biclosed subsets \cite[Theorem 5.4]{mcconville2014biclosed}, but in general these two collections differ.

As is true for Cambrian fans, the fan poset $(\Fcal_{\lambda},\GT(\lambda))$ associated to a shape $\lambda$ is a quotient lattice of the lattice of biclosed subsets of $\Acal_{\lambda}$. In the same way that a $W$-associahedron whose normal fan is a Cambrian fan may be constructed by removing facets of the $W$-permutahedron \cite[Theorem 3.4]{hohlweg2011permutahedra}, we conjecture that the Grid-associahedron may be constructed in the same manner.

\begin{conjecture}\label{conj_grid_associahedron}
There exists a polytope $P$ whose normal fan is $\Acal_{\lambda}$ such that the Grid-associahedron whose normal fan is $\Fcal_{\lambda}$ may be constructed by removing some facets from $P$.
\end{conjecture}

\subsection{Shard intersection order}\label{sec_shard}

The \emph{(geometric) shard intersection order} $\Psi^f(\lambda)$ is the poset of intersections of shards $\bigcap_{s\in I}\Sigma(s)$ for $I\subseteq\Seg(\lambda)$, ordered by \emph{reverse} inclusion. The poset $\Psi^f(\lambda)$ is a join-semilattice where the join of $Z,Z^{\pr}\in\Psi^f(\lambda)$ is equal to $Z\cap Z^{\pr}$. Since $\Psi^f(\lambda)$ has a bottom element, $\hat{0}=\bigcap_{s\in I}\Sigma(s),\ I=\emptyset$, this (finite) join-semilattice is a lattice.

We begin this section by proving that the set of shards containing a fixed point $x\in\Rbb^r$ is a wide set (Lemma~\ref{lem_wide_subset}). We divide the proof into several claims.

\begin{claim}\label{lem_inter_shards_1}
If $s,t,u$ are segments such that $s\circ t=u$, then
$$\Sigma(s)\cap\Sigma(t)=\Sigma(s)\cap\Sigma(u)=\Sigma(t)\cap\Sigma(u).$$
\end{claim}

\begin{proof}

Let $s,t,u$ be segments such that $s\circ t=u$. Without loss of generality, we may assume that $s\in A_u,\ t\in K_u$.

We first show that $\Sigma(s)\cap \Sigma(t)\subseteq \Sigma(u)$. Let $x\in\Sigma(s)\cap\Sigma(t)$. Then $\alpha_u(x)=\alpha_s(x)+\alpha_t(x)=0$.

If $u^{\pr}$ is a SW-subsegment of $u$, then $u^{\pr}=s^{\pr}\circ t^{\pr}$ where $s^{\pr}$ and $t^{\pr}$ are (possibly empty) SW-subsegments of $s$ and $t$, respectively. Hence, $\alpha_{u^{\pr}}(x)=\alpha_{s^{\pr}}(x)+\alpha_{t^{\pr}}(x)\geq 0$.

Similarly, if $u^{\pr}$ is a NE-subsegment of $u$ then $\alpha_{u^{\pr}}(x)\leq 0$. Consequently, $\Sigma(s)\cap \Sigma(t)\subseteq \Sigma(u)$.

Next we show that $\Sigma(s)\cap\Sigma(u)\subseteq\Sigma(t)$. Let $y\in\Sigma(s)\cap\Sigma(u)$. As before, we have $\alpha_t(y)=\alpha_u(y)-\alpha_s(y)=0$.

If $t^{\pr}$ is a NE-subsegment of $t$, then $t^{\pr}\in K_u$. Hence, $\alpha_{t^{\pr}}(y)\leq 0$.

If $t^{\pr}$ is a SW-subsegment of $t$, then it is either a SW-subsegment of $u$ or it contains the initial vertex of $t$. In the former case, we have $\alpha_{t^{\pr}}(y)\geq 0$ since $y\in\Sigma(u)$. In the latter case,

$$\alpha_{t^{\pr}}(y) = \alpha_s(y)+\alpha_{t^{\pr}}(y) = \alpha_{s\circ t^{\pr}}(y) \geq 0.$$

The last inequality follows from the fact that $s\circ t^{\pr}$ is a SW-subsegment of $u$. Hence, $\Sigma(s)\cap\Sigma(u)\subseteq\Sigma(t)$.

By a similar argument, we have $\Sigma(t)\cap\Sigma(u)\subseteq\Sigma(s)$. We conclude that
$$\Sigma(s)\cap\Sigma(t)=\Sigma(s)\cap\Sigma(u)=\Sigma(t)\cap\Sigma(u).$$\end{proof}

\begin{claim}\label{lem_inter_shards_2}
If $s_1,s_2,t,u$ are segments such that $u=s_1\circ t\circ s_2$ and either $t\in A_u$ or $t\in K_u$, then
$$\Sigma(u)\cap\Sigma(t)\subseteq\Sigma(s_1)\cap\Sigma(s_2).$$
\end{claim}

\begin{proof}

Let $s_1,s_2,t,u$ be segments such that $u=s_1\circ t\circ s_2$. We will assume that $t\in K_u$. We prove $\Sigma(u)\cap\Sigma(t)\subseteq\Sigma(s_1)\cap\Sigma(s_2)$. The case where $t$ is in $A_u$ may be proved in a dual manner.

By symmetry, it is enough to verify that $\Sigma(s_1)$ contains $\Sigma(t)\cap \Sigma(u)$. Let $x\in \Sigma(t)\cap \Sigma(u)$ be given. If $s^{\pr}$ is a SW-subsegment of $s_1$, then $s^{\pr}$ is a SW-subsegment of $u$, so $\alpha_{s^{\pr}}(x)\leq 0$. If $s^{\pr}$ is a NE-subsegment of $s_1$, then $s^{\pr}\circ t$ is a NE-subsegment of $u$, so $\alpha_{s^{\pr}}(x)=\alpha_{s^{\pr}\circ t}(x)-\alpha_t(x)\geq 0$. Hence, we have $\Sigma(t)\cap \Sigma(u)\subseteq \Sigma(s_1)$.\end{proof}

\begin{claim}\label{lem_inter_shards_3}
If $s$ and $t$ are friendly along a common subsegment $u$, then $\Sigma(s)\cap\Sigma(t)\subseteq\Sigma(u)$.
\end{claim}

\begin{proof}

Let $s,t$ be segments that are friendly along a common subsegment $u$; say $u\in A_s\cap K_t$. Let $x\in \Sigma(s)\cap \Sigma(t)$ be given. Then $\alpha_u(x)\geq 0$ since $u\in A_s$ and $\alpha_u(x)\leq 0$ since $u\in K_t$, which means $\alpha_u(x)=0$. Let $u^{\pr}$ be a NE-subsegment of $u$. Then $u^{\pr}\in K_t$, so $\alpha_{u^{\pr}}(x)\leq 0$. Similarly, if $u^{\pr}$ is a SW-subsegment of $u$, then $u^{\pr}\in A_s$ and $\alpha_{u^{\pr}}(x)\geq 0$. Hence, $\Sigma(s)\cap \Sigma(t)\subseteq \Sigma(u)$.\end{proof}

Combining Claims~\ref{lem_inter_shards_1}, \ref{lem_inter_shards_2}, and \ref{lem_inter_shards_3}, we deduce the following lemma.

\begin{lemma}\label{lem_wide_subset}
For a subset $X\subseteq\Rbb^r$, the set of shards containing $X$ is a wide set.
\end{lemma}

An element of the shard intersection order is defined as the intersection of some set of shards. In the following lemma, we give an equivalent definition as the union of some faces of the fan $\Fcal_{\lambda}$.

\begin{lemma}\label{lem_shard_intersection_faces}
For a set of segments $T\subseteq\Seg(\lambda)$, the cone $\bigcap_{s\in T}\Sigma(s)\in\Psi^f(\lambda)$ is equal to
$$\bigcup_{F\in\bigcap_{s\in T}\Delta_s}C(F).$$
\end{lemma}

\begin{proof}
Fix a set of segments $T$. Let $Z=\bigcap_{s\in T}\Sigma(s)$, and let
$$Z^{\pr}=\bigcup_{F\in\bigcap_{s\in T}\Delta_s}C(F).$$
Since $\Sigma(s)$ contains $C(F)$ for all $F\in\Delta_s$, it is clear that $Z$ contains $Z^{\pr}$.

Let $x\in Z$, and let $F$ be the smallest face of $\wtil{\Delta}^{NK}(\lambda)$ such that $x\in C(F)$. Then $x\in\Sigma(s)$ for any $s\in T$. Since $\Sigma(s)$ is defined as the union of faces in $\Delta_s$, $F$ must be a face of some $F^{\pr}\in\Delta_s$. Since $\Delta_s$ is a simplicial complex, this means that $F\in\Delta_s$. Hence, $x\in Z^{\pr}$, and the sets $Z$ and $Z^{\pr}$ are equal.
\end{proof}

\begin{lemma}\label{lem_shard_intersection_codim}
If $X$ is a nonfriendly set of segments, then $\bigcap_{s\in X}\Sigma(s)$ is a cone of codimension $|X|$.
\end{lemma}

\begin{proof}
Let $X$ be a nonfriendly set, and let $Z$ be the polyhedral cone $\bigcap_{s\in X}\Sigma(s)$. By Corollary~\ref{cor_lower_cover}, there exists a facet $F$ of $\wtil{\Delta}^{NK}$ such that $X=\{s:\ \exists F^{\pr}\stackrel{s}{\ra}F\}$. Since $(\Fcal_{\lambda},\GT(\lambda))$ is a simplicial fan poset, the intersection of the lower walls of $F$ is a face of codimension $|X|$. Since $Z$ contains a cone of codimension $|X|$, we have $\codim Z\leq|X|$. On the other hand, since $F$ is a simplicial cone, we have $\codim\bigcap_{s\in X}H_s=|X|$. Since $Z$ is contained in $\bigcap_{s\in X}H_s$, it follows that $Z$ has codimension $|X|$.
\end{proof}

We are now ready to state the main result of this section.

\begin{theorem}\label{thm_shard_intersection_order}
The geometric shard intersection order $\Psi^f(\lambda)$ is isomorphic to the lattice of wide sets $\Psi^w(\lambda)$.
\end{theorem}

\begin{proof}

Let $Z\in\Psi^f(\lambda)$ be given, and let $T=\{s\in\Seg(\lambda):\ Z\subseteq \Sigma(s)\}$. By Lemma~\ref{lem_wide_subset}, $T$ is a wide set. Clearly $Z\subseteq\bigcap_{s\in T}\Sigma(s)$ holds. Since $Z$ is the intersection of some shards and $T$ contains every segment $s$ such that $Z\subseteq\Sigma(s)$, it follows that $Z=\bigcap_{s\in T}\Sigma(s)$.

Now let $T^{\pr}$ be a given wide set, and put $Z^{\pr}=\bigcap_{s\in T^{\pr}}\Sigma(s)$. We prove that $T^{\pr}=\{s\in\Seg(\lambda):\ Z^{\pr}\subseteq\Sigma(s)\}$. This will complete the proof that there is a bijection between $\Psi^f(\lambda)$ and $\Psi^w(\lambda)$. It is clear that this bijection is order-preserving in both directions, so it defines an isomorphism between these two posets.

Let $u\in\Seg(\lambda)$ such that $Z^{\pr}\subseteq \Sigma(u)$. We show that $u\in T^{\pr}$ by induction on the length of $u$. Suppose $Z^{\pr}\subseteq\Sigma(t)$ implies $t\in T^{\pr}$ for any proper subsegment $t$ of $u$.

Fix some $x\in Z^{\pr}$. Suppose $u$ and $s$ are friendly at a common segment $t$ for some $s\in T^{\pr}$. We will assume that $t$ is a SW-subsegment of $u$. The proof for $t\in K_u$ is similar. By Claim~\ref{lem_inter_shards_3}, we have $Z^{\pr}\subseteq \Sigma(u)\cap \Sigma(s)\subseteq \Sigma(t)$. By the assumption on $u$, the segment $t$ is in $T^{\pr}$. There exist (possibly empty) NE-subsegments $t_1,t_2$ of $u$ such that $u=t_1\circ t\circ t_2$. If $u^{\pr}$ is in $K_{t_1}$ then $u^{\pr}$ is in $K_u$. Hence, $\alpha_{u^{\pr}}(x)\leq 0$. On the other hand, if $u^{\pr}\in A_{t_1}$, then $u^{\pr}\circ t$ is in $A_u$. In this case, $\alpha_{u^{\pr}}(x)=\alpha_{u^{\pr}\circ t}(x)-\alpha_t(x)\geq 0$. Hence, $Z^{\pr}\subseteq \Sigma(t_1)$ and $t_1\in T^{\pr}$ by induction. Similarly, $t_2\in T^{\pr}$. Since $T^{\pr}$ is closed, this implies $u\in T^{\pr}$.

Now assume that $u$ and $s$ are nonfriendly for all $s\in T^{\pr}$. As any segment is friendly with itself, we assume in particular that $u\notin T^{\pr}$. By Proposition~\ref{prop_nonfriendly_transitive}, $T^{\pr}$ is the closure of the nonfriendly set $\NF(T^{\pr})$, where
$$\NF(T^{\pr})=\{s\in T^{\pr}:\ \nexists t\in T^{\pr}\setm\{s\}\ \ t\in A_s \mbox{ or } t\in K_s\}.$$
By Claim~\ref{lem_inter_shards_1}, $\Sigma(s)\cap \Sigma(t)\subseteq \Sigma(u)$ whenever $u=s\circ t$. It follows that $Z^{\pr}=\bigcap_{s\in\NF(T^{\pr})}\Sigma(s)$. Since $\NF(T^{\pr})\cup\{u\}$ is a nonfriendly set, Lemma~\ref{lem_shard_intersection_codim} implies that
$$\codim Z^{\pr}<\codim Z^{\pr}\cap\Sigma(u).$$
Hence, $\Sigma(u)$ does not contain $Z^{\pr}$, contrary to our assumption.\end{proof}

\begin{lemma}\label{lem_shard_int_trans_face}
For a wide set $T$, if $C$ is a maximal face of $\Fcal_{\lambda}$ contained in $\bigcap_{t\in T}\Sigma(t)$, then $C$ is not contained in $\Sigma(s)$ for $s\notin T$.
\end{lemma}

\begin{proof}
Let $T$ be a wide set, and set $Z=\bigcap_{t\in T}\Sigma(t)$. From the proof of Theorem~\ref{thm_shard_intersection_order}, we have $Z=\bigcap_{t\in\NF(T)}\Sigma(t)$. Let $C$ be a maximal face of $\Fcal_{\lambda}$ contained in $Z$. Then $\codim C=|\NF(T)|$.

We claim that $C$ is not contained in $\Sigma(s)$ for any segment $s$ not in $T$. Suppose this is not the case, and let $s$ be of minimum length such that $s$ is not in $T$ and $C$ is contained in $\Sigma(s)$.

If $s$ is not friendly with any segment in $T$, then $\codim Z\cap\Sigma(s)>|\NF(T)|$, contradicting the assumption that $C$ is contained in $\Sigma(s)$. Hence, we may assume that $s$ is friendly with some segment $t\in T$.

Suppose $s$ and $t$ are friendly along a common subsegment $u$. Without loss of generality, we may assume that $u\in A_s\cap K_t$. By Claim~\ref{lem_inter_shards_3}, $\Sigma(s)\cap\Sigma(t)$ is contained in $\Sigma(u)$, so $C\subseteq\Sigma(u)$. Let $s_1,s_2$ be (possibly empty) segments such that $s=s_1\circ u\circ s_2$. By Claim~\ref{lem_inter_shards_2}, we have $\Sigma(s)\cap\Sigma(u)\subseteq\Sigma(s_1)\cap\Sigma(s_2)$, which means $C\subseteq\Sigma(s_1)$ and $C\subseteq\Sigma(s_2)$. By the minimality assumption on $s$, the segments $u,s_1,s_2$ are all in $T$. Since $T$ is a closed set, it must contain $s$, a contradiction.\end{proof}

By Lemma~\ref{lem_shard_int_trans_face}, every face $C\in\Fcal_{\lambda}$ is a maximal face of a unique $Z$ in $\Psi^f(\lambda)$. Consequently, if $Z,Z^{\pr}\in\Psi^f(\lambda)$ such that $Z^{\pr}<Z$, then $\codim Z^{\pr}<\codim Z$. This is a key result to proving the following statement.

\begin{proposition}\label{prop_shard_int_graded_lattice}
The poset $\Psi^f(\lambda)$ is a graded lattice.
\end{proposition}

\begin{proof}
As we remarked at the beginning of this section, the intersection $Z\cap Z^{\pr}$ is the join of $Z$ and $Z^{\pr}$ in $\Psi^f(\lambda)$. Since $\Psi^f(\lambda)$ is a finite poset with a minimum element, it follows that it is a lattice.

We show that $\Psi^f(\lambda)$ is graded by codimension. Let $Z,Z^{\pr}\in\Psi^f(\lambda)$ such that $Z^{\pr}\leq Z$. Let $C$ be a face of $\Fcal_{\lambda}$ such that $C$ is a maximal face of $Z$. Since $C\subseteq Z^{\pr}$ and $Z^{\pr}$ is equal to the union of its maximal faces, there exists a face $C^{\pr}\in\Fcal_{\lambda}$ maximal in $Z^{\pr}$ such that $C\subseteq C^{\pr}$. If $C\neq C^{\pr}$, then there exists a face $C_1$ such that $C\subsetneq C_1\subseteq C^{\pr}$ and $\codim C_1=\codim C-1$. Let $Z_1\in\Psi^f(\lambda)$ such that $C_1$ is a maximal face in $Z_1$. Since every shard containing $Z^{\pr}$ also contains $C_1$, we have $Z^{\pr}\leq Z_1$. Conversely, $Z_1<Z$ holds. By induction, there exists a chain from $Z^{\pr}$ to $Z$ with $\codim Z-\codim Z^{\pr}+1$ elements. By the discussion preceeding this proof, this chain is maximal.
\end{proof}

\section{Enumeration}\label{sec_enumeration}

\subsection{$f$-vector and $h$-vector}\label{subsec_fhpoly}

Given a simplicial complex $\Delta$, let $f_d$ be the number of $d$-dimensional faces of $\Delta$ for each $d\in\{-1,0,1,\ldots\}$. Every nonvoid simplicial complex contains the \emph{empty face}, which is the unique face of dimension $-1$. The \emph{$f$-vector} $(f_{-1},f_0,f_1,\ldots)$ is the sequence of face numbers of $\Delta$. The \emph{$f$-polynomial} is
$$f(t)=\sum_df_{d-1}t^d.$$

If the largest face is of dimension $r-1$, then the \emph{$h$-vector} $(h_0,h_1,h_2,\ldots,h_r)$ is the sequence of integers defined by the identity
$$\sum_{d=0}^r f_{d-1}(t-1)^{r-d}=\sum_{d=0}^r h_dt^{r-d}.$$
The polynomial $h(t)=\sum_{d=0}^r h_dt^{r-d}$ is the \emph{$h$-polynomial} of $\Delta$. The above identity may be compactly expressed as $h(t+1)=t^rf(1/t)$. For example, if $\Delta$ is the reduced nonkissing complex for a $2\times 3$ rectangle, then

$$f(t)=1+5t+5t^2,\ \hspace{5mm} h(t)=(t-1)^2 + 5(t-1) + 5=1+3t+t^2.$$

A pure simplicial complex $\Delta$ is \emph{shellable} if its facets may be totally ordered as $F_1, F_2,\ldots$ such that for all $1\leq i<j$, there exists $k<j$ such that $F_k\cap F_j$ is a ridge that contains $F_i\cap F_j$. If $\Delta$ is shellable, then for each index $j$, there exists a unique minimal face $R(F_j)$ of $F_j$ not contained in $\bigcup_{i<j}F_i$. In this situation, the $h$-polynomial is equal to $\sum_j x^{|R(F_j)|}$.

Reading proved that if $(\Fcal,P)$ is a simplicial fan poset, then any linear extension of $P$ is a shelling order on the facets of $\Fcal$ \cite[Proposition 3.5]{reading:lattice_congruence}. Furthermore, for each facet $F\in\Fcal$, $|R(F)|$ is equal to the number of lower covers $F^{\pr}\lessdot F$. Hence, Theorem~\ref{thm_fan_order} implies the first assertion of the following result. The second assertion is deduced from Proposition~\ref{prop_ineq_fan}.

\begin{lemma}\label{lem_shellable}
Any linear extension of $\GT(\lambda)$ is a shelling order on the facets of the reduced nonkissing complex. Moreover, the $h$-polynomial of the reduced nonkissing complex equals the $f$-polynomial of the nonfriendly complex.
\end{lemma}

\subsection{$F$-triangle and $H$-triangle}\label{subsec_triangle_FH}

We recall the $F$-triangle and $H$-triangle from Section~\ref{sec_nonkissing}. Fix a shape $\lambda$, and let $V^o$ be the set of interior vertices. For a boundary path $p$, recall that the $g$-vector $g_p\in\Rbb^{V^o}$ (Section~\ref{subsec_fan_nonkissing}) is defined as
$$g_p(v)=\begin{cases}1\ &\mbox{if }p\mbox{ enters }v\mbox{ from the North and leaves to the East,}\\-1\ &\mbox{if }p\mbox{ enters }v\mbox{ from the West and leaves to the South,}\\0\ &\mbox{otherwise.}\end{cases}$$
We say a path $p$ is \emph{non-initial} if $g_p(v)=1$ for some interior vertex $v$. Otherwise, we say $p$ is \emph{initial}. There is a distinguished facet $F_0$ in $\wtil{\Delta}^{NK}$ that consists of all of the initial boundary paths. The \emph{positive part} of the reduced nonkissing complex $\wtil{\Delta}^{NK}_+$ is the full subcomplex of $\wtil{\Delta}^{NK}$ on the non-initial paths.

Let $r=|V^o|$, and label the interior vertices $v_1,\ldots,v_r$. For $i\in[r]$, let $t_i$ be the lazy segment supported at $v_i$, and let $q_i$ be the initial boundary path that turns at $v_i$. The $F$-triangle introduced in Section~\ref{subsec_comb_nonkissing} is the polynomial

$$F(x,y)=\sum_{F\in\wtil{\Delta}^{NK}}x^{|F\setm F_0|}y^{|F\cap F_0|}=\sum_{i,j}f_{ij}x^{i-j}y^j.$$

We remark that $f_{ij}=0$ unless $0\leq j\leq i\leq r$ where $r=|V^o|$. We consider a multivariate extension of this polynomial

$$F(x,y_1,\ldots,y_r) = \sum_{F\in\wtil{\Delta}^{NK}}x^{|F\setm F_0|}\prod_{q_i\in F}y_i.$$

It is clear that this polynomial extends the $F$-triangle in the sense that $F(x,y)=F(x,y,\ldots,y)$.

If $\Gamma=\Gamma(\lambda)$ is the nonfriendly complex, the $H$-triangle introduced in Section~\ref{subsec_comb_nonfriendly} is the polynomial

$$H(x,y)=\sum_{F\in\Gamma}x^{|F|}y^{|\epsilon(F)|}=\sum_{i,j}h_{ij}x^iy^j,$$

\noindent where $\epsilon(F)$ is the set of isolated lazy segments in $F$. Once again, $h_{ij}=0$ unless $0\leq j\leq i\leq r$ holds. We observe that the column sums agree with the $f$-vector and $h$-vector. That is, $f_j=\sum_{i=0}^jf_{ij}$ and $h_j=\sum_{i=0}^jh_{ij}$ for all $j$.

As before, there is a multivariate extension of the $H$-triangle to

$$H(x,y_1,\ldots,y_r)=\sum_{F\in\Gamma}x^{|F|}\prod_{t_i\in\epsilon(F)}y_i.$$

As mentioned in Section~\ref{subsec_comb_paths_segments}, we may declare that some of the vertices of degree $4$ are boundary vertices. Given a subset $B$ of $V^o$, we consider a shape $\lambda^{\pr}$ which is the same graph as $\lambda$ except that the vertices in $B$ are viewed as boundary vertices. This means that boundary paths for $\lambda^{\pr}$ may begin or end at a vertex in $B$, and segments may not include any vertex in $B$. All of the results about the nonkissing complex and the Grid-Tamari order still hold for $\lambda^{\pr}$.

\begin{theorem}\label{thm_FH}
The following identity holds.
\begin{align}
\label{thm_FH_eqn} H(x+1,\ y_1+1,\ \ldots,y_r+1) = x^r F\left(\frac{1}{x},\ \frac{1+y_1(x+1)}{x},\ \ldots,\ \frac{1+y_r(x+1)}{x}\right)
\end{align}
\end{theorem}

\begin{proof}

Let $\Delta=\wtil{\Delta}^{NK}(\lambda)$ be the reduced nonkissing complex. Given a subset $I\subseteq [r]$, we let $\lk_I(\Delta)$ be the link of the face $\{q_i:\ i\in I\}$ of $\Delta$. The positive part $\lk_I(\Delta)_+$ is the subcomplex obtained by deleting all initial boundary paths from the link. Then $\Delta$ decomposes as

$$\Delta=\bigsqcup_{I\subseteq [r]}\{\{q_i:\ i\in I\}\cup F:\ F\in\lk_I(\Delta)_+\}.$$

Using this decomposition of $\Delta$, we may reduce the right-hand side of the above identity as follows:

\begin{align*}
x^r F\left(\frac{1}{x},\ \frac{1+y_1(x+1)}{x},\ \ldots,\ \frac{1+y_r(x+1)}{x}\right) &= \sum_{F\in\Delta}x^{r-|F|}\prod_{q_i\in F}(1+y_i(x+1))\\
&= \sum_{J\subseteq[r]}\sum_{F\in\lk_J(\Delta)_+}x^{r-|J|-|F|}\prod_{i\in J}(1+y_i(x+1))\\
&= \sum_{I\subseteq[r]}(x+1)^{|I|}\sum_{I\subseteq J\subseteq[r]}\sum_{F\in\lk_J(\Delta)_+}x^{r-|J|-|F|}\prod_{i\in I}y_i\\
&= \sum_{I\subseteq[r]}(x+1)^{|I|}f_I^{\pr}(x)\prod_{i\in I}y_i
\end{align*}

\noindent where $f_I^{\pr}(x)$ is defined as
$$f_I^{\pr}(x)=\sum_{I\subseteq J\subseteq[r]}\sum_{F\in\lk_J(\Delta)_+}x^{r-|J|-|F|}.$$

Using the decomposition
$$\lk_I(\Delta)=\bigsqcup_{I\subseteq J\subseteq[r]}\{\{q_j:\ j\in J\}\cup F:\ F\in\lk_J(\Delta)_+\},$$

\noindent we have

$$f_I^{\pr}(x)=\sum_{F\in\lk_I(\Delta)}x^{(r-|I|)-|F|},$$

\noindent so the polynomial $x^{r-|I|}f_I^{\pr}(1/x)$ is the $f$-polynomial of $\lk_I(\Delta)$.

Now expand the left-hand side:

\begin{align*}
H(x+1,\ y_1+1,\ \ldots,\ y_r+1) &= \sum_{F\in\Gamma}(x+1)^{|F|}\prod_{t_i\in\epsilon(F)}(1+y_i)\\
&= \sum_{I\subseteq[r]}\sum_{\substack{F\in\Gamma\\I\subseteq\epsilon(F)}}(x+1)^{|F|}\prod_{i\in I}y_i.
\end{align*}

In the above expression, we write $I\subseteq\epsilon(F)$ to mean that $\{t_i:\ i\in I\}$ is a subset of $\epsilon(F)$.

Now fix some $I\subseteq[r]$. Let $\lambda^{\pr}$ be the same shape as $\lambda$ where each of the vertices $v_i$ for $i\in I$ are treated as boundary vertices. The nonfriendly complex $\Gamma(\lambda^{\pr})$ consists of those collections of segments in $\Gamma(\lambda)$ for which no segment contains a vertex $v_i$ for some $i\in I$. We remark that a lazy segment $t_i$ and some other segment $s$ are nonfriendly if and only if $v_i$ is not in $s$. In particular, $\Gamma(\lambda^{\pr})$ is equal to the subcomplex of the nonfriendly complex $\Gamma(\lambda)$ consisting of faces $F$ disjoint from $I$ such that $F\cup I\in\Gamma$ and $I\subseteq\epsilon(F\cup I)$. In particular, we have

\begin{align*}
H(x+1,\ y_1+1,\ \ldots,\ y_r+1) &= \sum_{I\subseteq[r]}(x+1)^{|I|}h_I(x+1)\prod_{i\in I}y_i,
\end{align*}

\noindent where $h_I(x)$ is the $h$-polynomial of $\lk_I(\Delta)$. The now follows from Lemma~\ref{lem_shellable}.\end{proof}

Setting $y=y_1=\cdots=y_r$, we obtain the immediate corollary.

\begin{corollary}
$$H(x+1,y+1)= x^r F\left(\frac{1}{x},\ \frac{1+y(x+1)}{x}\right)$$
\end{corollary}

\subsection{$F$-triangle and $M$-triangle}\label{subsec_FM}

Let $\Psi=\Psi^f(\lambda)$ be the shard intersection order for the shape $\lambda$. The $M$-triangle (introduced in Section~\ref{subsec_transitive}) is the polynomial

$$M(x,y)=\sum_{\substack{X,Y\in\Psi\\Y\leq X}}\mu(Y,X)x^{\rk X}y^{\rk Y}=\sum_{i,j=0}^rm_{ij}x^iy^j.$$

This polynomial is well-defined since $\Psi$ is graded by Proposition~\ref{prop_shard_int_graded_lattice}. We conjecture the following identity between the $M$-triangle and $F$-triangle.

\begin{conjecture}\label{conj_MF}
$$M(-x,-y/x) = (1-y)^rF\left(\frac{x+y}{1-y},\frac{y}{1-y}\right)$$
\end{conjecture}

This conjecture has been verified using Sage \cite{sage} for the $3\times 4$ rectangle shape, along with its subgraphs.

\begin{example}

Let $\lambda$ be a $3\times 4$ rectangle shape. We represent the $F$-triangle, $H$-triangle, and $M$-triangle for this shape by three $7\times 7$ lower-triangular matrices as follows.

If $f_{ij}$ is the coefficient of $x^{i-j}y^j$ in $F(x,y)$, then

$$(f_{ij}) = \left(\begin{matrix}1 & 0 & 0 & 0 & 0 & 0 & 0\\22 & 6 & 0 & 0 & 0 & 0 & 0\\141 & 82 & 15 & 0 & 0 & 0 & 0\\395 & 344 & 123 & 20 & 0 & 0 & 0\\548 & 620 & 319 & 94 & 15 & 0 & 0\\371 & 506 & 332 & 134 & 37 & 6 & 0\\98 & 154 & 121 & 60 & 22 & 6 & 1\end{matrix}\right).$$

Letting $h_{ij}$ and $m_{ij}$ be the coefficients of $x^iy^j$ in $H(x,y)$ and $M(x,y)$, respectively, we have

$$(h_{ij}) = \left(\begin{matrix}1 & 0 & 0 & 0 & 0 & 0 & 0\\16 & 6 & 0 & 0 & 0 & 0 & 0\\46 & 52 & 15 & 0 & 0 & 0 & 0\\31 & 76 & 63 & 20 & 0 & 0 & 0\\4 & 20 & 40 & 34 & 15 & 0 & 0\\0 & 0 & 3 & 6 & 7 & 6 & 0\\0 & 0 & 0 & 0 & 0 & 0 & 1\end{matrix}\right),\ \mbox{and}$$

$$(m_{ij}) = \left(\begin{matrix}1 & 0 & 0 & 0 & 0 & 0 & 0\\-22 & 22 & 0 & 0 & 0 & 0 & 0\\141 & -254 & 113 & 0 & 0 & 0 & 0\\-395 & 965 & -760 & 190 & 0 & 0 & 0\\548 & -1627 & 1726 & -760 & 113 & 0 & 0\\-371 & 1265 & -1627 & 965 & -254 & 22 & 0\\98 & -371 & 548 & -395 & 141 & -22 & 1\end{matrix}\right).$$

It is routine to check that these three triangles satisfy the identities in Theorem~\ref{thm_FH} and Conjecture~\ref{conj_MF}.

\end{example}

\section*{Acknowledgements}

We thank Christophe Hohlweg for suggesting Conjecture~\ref{conj_grid_associahedron}. Thomas McConville thanks Alex Postnikov for suggesting references on flow polytopes. Alexander Garver received support from NSERC and the Canada Research Chairs program.

\bibliography{bib_tamari_v2}{}
\bibliographystyle{plain}

\end{document}